\documentclass[12pt,leqno]{article}
\usepackage{amssymb,amsmath,amsthm}
\usepackage[all]{xy}
\usepackage[alphabetic,lite]{amsrefs}
\numberwithin{equation}{section}

\usepackage[top=2cm,bottom=2cm,left=2cm,right=4cm,marginparsep=0.3cm,marginparwidth=3cm,includefoot]{geometry}
\usepackage{mathtools}
\usepackage{enumerate}
\usepackage{mathrsfs}
\usepackage{Bdist}

\begin{document}

\title
{Thickening of the diagonal and   interleaving distance}
\author{Fran{\c c}ois Petit and Pierre Schapira}
\maketitle
\begin{abstract}
Given a topological space $X$, a thickening kernel is a monoidal presheaf on $(\R_{\geq0},+)$ with values in the monoidal category of derived kernels on $X$. A bi-thickening kernel is defined on $(\R,+)$. To such a thickening kernel, one naturally associates an  interleaving distance  on  the derived category  of sheaves on $X$. 

We prove that a thickening kernel exists and is unique as soon as it is defined on an  interval containing $0$,  allowing  us to construct (bi-)thickenings  in two different situations.  

First,  when $X$ is a ``good'' metric space, starting  with small usual thickenings of the diagonal. 
The associated  interleaving  distance satisfies the stability property and Lipschitz kernels give rise to Lipschitz maps.  

Second, by using~\cite{GKS12}, when $X$ is a manifold and one is given a non-positive Hamiltonian isotopy on the cotangent bundle. In case $X$ is a  complete Riemannian manifold  
having a strictly positive convexity radius, we prove that it is a good metric space and  
that the two bi-thickening kernels of the diagonal, one associated with the distance, the other with the geodesic flow, coincide. 
 \end{abstract}

{\renewcommand{\thefootnote}{\mbox{}}
\footnote{F.P  was supported by the IdEx Université de Paris, ANR-18-IDEX-0001}
\addtocounter{footnote}{-4}
}
  
\tableofcontents

\section*{Introduction}
The aim of this paper is to construct (and then to study) a kernel associated with  a small thickening of the diagonal of a space $X$ and, as a byproduct, 
an interleaving distance  on the derived category of sheaves on  $X$. Such a kernel is constructed  in essentially two rather different situations:  first  when $X$ is a metric space by using the distance,
second when $X$ is a manifold and one is given a non positive Hamiltonian isotopy of the cotangent bundle. When $X$ is a Riemannian manifold and the isotopy is associated with the geodesic flow, we prove that the two kernels coincide.

The interleaving distance introduced by F. Chazal et al. in \cites{FCMGO09} has become a central element of TDA and has been actively studied since then \cites{BBK18,BP18,BL17,BG18}. It was generalised to multi-persistence modules by M. Lesnick in \cites{Les12, Les15}. Categorical frameworks for the interleaving distance have then been proposed  in~\cites{Bub14, SMS18}.
In his thesis \cite{Cur14}, J. Curry proposed an approach of persistence homology via sheaf theory. In \cite{KS18}, the author developed derived sheaf-technics for persistent homology and defined a new interleaving distance
for the category of {\em derived} sheaves on a real normed vector space by considering thickenings associated with the convolution by closed balls of radius $a\geq0$. This distance is sometimes called the convolution distance for sheaves and has recently been applied to question of symplectic topology (see for instance \cite{Ike17}). For a survey of the links between the ($1$-dimensional) interleaving distance, sheaf theory and symplectic topology, see the book by J. Zhang~\cite{JZ18}.

\spbb
Let $X$ be a ``good''  topological space  and denote as usual by $\Derb(\cor_X)$ the bounded derived  category  of sheaves of $\cor$-modules on $X$, for a 
commutative unital ring of finite global dimension $\cor$. We define a  thickening kernel on $X$ as
a monoidal presheaf $\stK$ defined on the monoidal category $(\R_{\geq0},+)$ with values in the monoidal category  
$(\Derb(\cor_{X\times X}),\conv)$ of kernels on $X$ (see Definition~\ref{def:pshK}). When this presheaf extends as a monoidal presheaf on $(\R,+)$, we call it a bi-thickening kernel of the diagonal. 

To a thickening kernel,  one naturally associates an interleaving distance $\dist_X$ on $\Derb(\cor_X)$. 

Our first result (Theorem~\ref{th:pshK}) asserts that  a thickening kernel exists  and is unique (up to isomorphism) as soon as it is constructed on some interval $[0,\alpha_X]$ (with $\alpha_X>0$).

\spbb
This theorem allows us to construct a (bi-)thickening kernel in two different situations. First in~\S~\ref{sect:metric}, 
when $X$ is what we call here a {\em good metric space} (see Definition~\ref{def:goodmetricsp}).
Second  in~\S~\ref{sect:symp}, when $X$ is a real manifold and one is given a non-positive $C^\infty$-function $h\cl\dT^*X\to\R$, where $\dT^*X$ is  the cotangent bundle with the zero-section removed.

\spbb
(1) Assume that  $(X,d_X)$ is a good metric space and denote by $\Delta_a$ the closed thickening of radius $a\geq0$ of the diagonal.
The hypothesis that $(X,d_X)$ is good  implies in particular that $\cor_{\Delta_a}\conv \cor_{\Delta_b}\simeq \cor_{\Delta_{a+b}}$  for $a,b$ sufficiently small. Applying our first theorem, we get a thickening kernel $\stK$ on $(\R_{\geq0},+)$ or, under mild extra-hypotheses, a bi-thickening.   In this case,  for $a<0$ small,
$\stK_a$   is, up to a shift and an orientation, the kernel associated with an open thickening of the diagonal.

We obtain several results on the associated  interleaving distance, some of them  generalizing  those of~\cite{KS18}. 
We prove in particular a stability theorem (Theorem~\ref{th:stabmetric}) which asserts that given two kernels $K_1$ and $K_2$ on $Y\times X$ and a sheaf $F$ on $X$, then 
$\dist_Y(K_1\conv F,K_2\conv F)\leq \dist_{Y\times X/X}(K_1,K_2)$ where $ \dist_{Y\times X/X}$ is a relative distance. 
We also introduce the notion of  a {\em $\delta$-Lipschitz kernel} on $Y\times X$ and show that such a kernel induces a Lipschitz map for the interleaving distances (Theorem~\ref{th:FctLip}). In both cases (stability and Lipschitz) we also obtain similar results for non proper composition, but then we need to assume that our spaces are  manifolds and the differential of the distance does not vanish. Indeed, in this situation, our proofs are based on Theorem~\ref{th:assocnp} which asserts that under some microlocal hypotheses,  non proper composition becomes associative.

\spbb
(2) Assume now that  $X$ is a real manifold and one is given a  $C^\infty$-function $h\cl\dT^*X\to\R$, homogeneous of degree $1$ in the fiber such that the flow $\Phi$ of the Hamiltonian vector field of $h$ is an Hamiltonian isotopy 
defined on  $\dT^*X\times I$ for some open interval $I$ containing $0$. This flow  gives rise to a Lagrangian manifold 
$\Lambda\subset \dT^*X\times \dT^*X\times T^*I$. Thanks to the main  theorem of~\cite{GKS12},
there exists a unique kernel $K^h\in\Derlb(\cor_{X\times X\times I})$ micro-supported by $\Lambda$ and whose restriction to $t=0$ is $\cor_\Delta$. Moreover, since $h$  is not time depending, this kernel satisfies $K^h_a\conv K^h_b\simeq K^h_{a+b}$ for $a,b$ small. Assuming $h$ is non-positive, there are natural morphisms $K^h_b\to K^h_a$ for 
$a\leq b$ and using our first theorem we get a bi-thickening kernel $\stK^h$.

When $X$ is a complete  Riemannian manifold having a strictly positive convexity radius, we prove 
(Theorem~\ref{th:riemann1}) that  it is  a good metric space and the associated thickening kernel is a bi-thickening, denoted here $\stKd$. We have thus two bi-thickening kernels in this case, $\stKd$ and $\stK^h$, the last one being associated with 
the geodesic flow (corresponding to $h(x,\xi)=-\vvert\xi\vvert_x$). We prove in Theorem~\ref{th:FS2} that these two kernels coincide. 

\spbb
In the course of the paper, we treat some easy examples and in particular we prove that the Fourier-Sato transform, an equivalence of categories  for sheaves on a  sphere and the dual sphere,  is an isometry when endowing these spheres with their natural Riemannian metric. Indeed,  the Fourier-Sato transform is nothing but  the value at $\pi/2$ of the thickening kernel of the  Riemannian sphere.  

\medskip\noindent{\bf Acknowledgments}

The author  F.P. warmly thanks Vincent Pecastaing and Yannick Voglaire for fruitful comments. The author P.S  warmly thanks  Beno{\^i}t Jubin for the same reason.
Both authors warmly thank St{\'e}phane Guillermou  for extremely valuable remarks and also for 
his  proof of Lemma~\ref{le:gu} which considerably simplifies an earlier proof. 

\section{Sheaves and the interleaving distance}

\subsection{Sheaves}
In the sequel, we denote by $\rmpt$ the topological space with a single element. For a topological space $X$, we denote by $a_X\cl X\to \rmpt$ the unique map fromt $X$ to $\rmpt$. We denote by $\Delta_X$, or simply $\Delta$, the diagonal of $X\times X$ and by $\delta_X$ or simply $\delta$ the diagonal embedding. If $X$ is a $C^\infty$-manifold, we denote by $\pi_X\cl T^*X\to X$ its cotangent bundle and by $\dT^*X$ the cotangent bundle with the zero-section removed. 
Recall that a topological  space $X$ is good if it is Hausdorff, locally compact, countable at infinity and of finite flabby dimension.

We consider  a commutative unital ring of finite global dimension $\cor$
and a good topological space $X$. We denote by $\RD(\cor_X)$ the derived category of sheaves of $\cor$-modules on $X$ and simply call an object of this category ``a sheaf''. We shall almost always work in the bounded derived category $\Derb(\cor_X)$ but 
we shall also need to consider  the full  subcategory $\Derlb(\cor_X)$ of $\RD(\cor_X)$
consisting of locally bounded objects, that is, objects  whose restriction to any relatively compact open subset $U$ of $X$ belong to $\Derb(\cor_U)$ (see~\cite{GKS12}*{Def.~1.12}).

We shall freely make use of the six Grothendieck operations on sheaves and refer to~\cite{KS90}. 
In particular, we denote by
$\omega_X$ the dualizing complex and we use the duality functors
\eqn
&&\RD'_X=\rhom(\scbul,\cor_X),\quad \RD_X=\rhom(\scbul,\omega_X).
\eneqn
For a locally closed subset $A\subset X$, we denote by $\cor_{XA}$ the sheaf on $X$ which is the constant sheaf with stalk $\cor$ on $A$ and $0$ elsewhere. If there is no risk of confusion, we simply denote it by $\cor_A$. If $F$ is a sheaf on $X$, one sets $F_A\eqdot F\tens\cor_A$.   We also often simply denote by $F\tens L$ the derived tensor product when $L$ is of the type $\cor_A$ up to a shift or an orientation. As usual, we denote by $\rsect(X;\scbul)$ and $\rsect_c(X;\scbul)$ the derived functors of global sections and global sections with compact supports.

When $X$ is a $C^\infty$-manifold, we shall  make use of the microlocal  theory of sheaves, following~\cite{KS90}*{Ch. V-VI}. 
Recall that the {\em micro-support} $\SSi(F)$ of a sheaf $F$ is a closed $\R^+$-conic subset of $T^*X$,  co-isotropic for the homogeneous symplectic structure of $T^*X$ (we shall not use here this property).  We shall also use the notation 
$\dot\SSi(F)\eqdot\SSi(F)\cap\dT^*X$.
We shall also encounter {\em cohomologically constructible} sheaves for which we refer to loc.\ cit.\  \S~3.4.
Recall that, on a  real analytic manifold,  $\R$-constructible sheaves   (see~loc.\ cit.\ Ch.~VIII) are  cohomologically constructible.

\subsubsection*{Kernels}
Given topological spaces $X_i$ ($i=1,2,3$) we set
$X_{ij}=X_i\times X_j$, $X_{123}=X_1\times X_2\times X_3$.    We denote by $q_i\cl X_{ij}\to X_i$
and $q_{ij}\cl X_{123}\to X_{ij}$ the projections.

We shall often write for short $\RD_i$ instead of $\RD_{X_i}$, as well as for similar notations such as for example  $\RD'_{i}$ or
$\RD_{ij}$.

For $A\subset X_{12}$ and $B\subset X_{23}$ one sets $A\conv B=q_{13}(\opb{q_{12}}A\cap\opb{q_{23}}B)$:
\eq\label{diag:123}
&&\ba{l}\xymatrix{
	&X_{123}\ar[ld]_{q_{12}}\ar[rd]^-{q_{23}}\ar[d]^-{q_{13}}&\\
	X_{12}&X_{13}&X_{23} . 
}\ea\eneq
When the spaces $X_i$'s are real manifolds, 
one denotes by $p_{ij}\cl T^*X_{123}\to T^*X_{ij}$ the projection and we also define
\eqn
&&p_{i^aj}\cl T^*X_{123}\to T^*X_{ij},\quad (x_1,x_2,x_3;\xi_1,\xi_2,\xi_3)\mapsto (x_i,x_j;-\xi_i,\xi_j)
\eneqn
the composition of $p_{ij}$ with the antipodal map of $T^*X_i$. 

For $A\subset T^*X_{12}$ and $B\subset T^*X_{23}$ one sets 
\eqn
&&A\aconv B=p_{13}(\opb{p_{12}}A\cap\opb{p_{2^a3}}B)
\eneqn

For good topological spaces $X_i$'s as above, one often calls an object $K_{ij}\in\Derb(\cor_{X_{ij}})$ {\em a kernel}.
One defines  as usual the  composition of kernels 
\eq\label{eq:convker}
&&K_{12}\cconv[2] K_{23}\eqdot \reim{q_{13}}(\opb{q_{12}}K_{12}\ltens\opb{q_{23}}K_{23}).
\eneq
If there is no risk of confusion, we write $\conv$ instead of $\cconv[2]$. 

It is sometimes natural to permute the roles of $X_i$ and $X_j$.
We introduce the notation
\eq\label{eq:volte}  
&&\ba{l}
v\cl X_{12}\to X_{21},\quad (x_1,x_2)\mapsto (x_2,x_1),\\
\nu \cl X_{123}\to  X_{321}, (x_1,x_2,x_3) \mapsto (x_3,x_2,x_1).
\ea \eneq
Since $v$ and $\nu$ are involutions, one has 
\eq\label{eq:oimv}
&&\oim{v}\simeq\eim{v},\, \opb{v}\simeq\epb{v},\quad \oim{\nu}\simeq\eim{\nu},\,\opb{\nu}\simeq\epb{\nu}. 
\eneq
Using~\eqref{eq:oimv}, one immediately obtains:
\begin{proposition}
	Let $K_{ij} \in\Derb(\cor_{X_{ij}})$, $i=1,2$, $j=i+1$ and set $K_{ji}\eqdot\oim{v}K_{ij}$. 
	Then
	\eqn
	&&\oim{v}(K_{12}\cconv[2] K_{23})\simeq K_{32}\cconv[2] K_{21}.
	\eneqn
\end{proposition}

In the sequel, we shall need to control the micro-support of the composition. Let $X_i$ and $K_{ij}$ be as above $i=1,2$, $j=i+1$.
Let  $A_{ij}=\SSi(K_{ij})\subset T^*X_{ij}$  and assume that  
\eq\label{eq:noncharker}
&&\left\{
\parbox{75ex}{
	(i) $q_{13}$ is proper on $\opb{q_{12}}\supp(K_{12})\cap\opb{q_{23}}\supp(K_{23})$,
	\\[1ex]
	(ii) 
	$\ba[t]{l}
	\opb{p_{12}}A_{12}\cap\opb{p_{2^a3}}A_{23}
	\cap (T^*_{X_1}X_1\times T^*X_2\times T^*_{X_3}X_3) \subset T^*_{X_{123}}(X_{123}).
	\ea$
}\right.
\eneq
\begin{proposition}\label{pro:GuS14}
	Assume~\eqref{eq:noncharker}. Then  
	\eq\label{eq:convolution_of_kern}
	&&\SSi(K_{12}\cconv[2] K_{23})\subset A_{12}\aconv A_{23}.
	\eneq
\end{proposition}
\begin{proof}
	This follows from the classical bounds to the micro-supports of proper direct images and non-characteristic inverse images of~\cite{KS90}*{\S~5.4}.
\end{proof}
The next lemma will be useful.
\begin{lemma}\label{le:kAcB}
	Let $A\subset X_{12}$ and $B\subset X_{23}$ be two closed subsets. 
	\anum
	\item
	Assume that $q_{13}$ is proper on $A\times_{X_2}B\eqdot \opb{q_{12}}A\cap\opb{q_{23}}B$. Then there is a natural morphism $\cor_{A\conv B}\to\cor_A\conv\cor_B$.
	\item
	Assume moreover that the fibers of the map $q_{13}\cl A\times_{X_2} B\to A\conv B$ are contractible. Then 
	$\cor_{A\conv B}\isoto\cor_A\conv\cor_B$.
	\eanum
\end{lemma}
\begin{proof}
	(a) Set $C=\opb{q_{12}}A\cap\opb{q_{23}}B$. Then $q_{13}(C)=A\conv B$ and 
	$\cor_C\simeq\opb{q_{12}}\cor_A\tens\opb{q_{23}}\cor_B$. By the hypothesis, the set $\opb{q_{13}}q_{13}(C)$ is closed and contains $C$.  Therefore, the morphism 
	$\opb{q_{13}}\cor_{q_{13}(C)}\to \cor_C$ defines by adjunction the morphism 
	$\cor_{A\conv B}\to\roim{q_{13}}(\opb{q_{12}}\cor_A\tens\opb{q_{23}}\cor_B)\isofrom\cor_A\conv\cor_B$ (recall that $q_{13}$ is proper on $C$). 
	
	\spa
	(b) is clear. 
\end{proof}

It is easily checked, and well-known, that the composition  of kernels  is associative, namely given three kernels $K_{ij} \in\Derb(\cor_{X_{ij}})$, $i=1,2,3$, $j=i+1$ one has 
an isomorphism 
\eq\label{eq:assockernels}
&&(K_{12}\cconv[2] K_{23})\cconv[3] K_{34}\simeq K_{12}\cconv[2] (K_{23}\cconv[3] K_{34}),
\eneq
this isomorphism satisfying natural compatibility conditions that we shall not make here explicit.

Of course, this construction applies in the particular cases where $X_i=\rmpt$ for some $i$.
For example, if $K\in\Derb(\cor_{X_{12}})$ and $F\in\Derb(\cor_{X_2})$, one usually sets $\Phi_K(F)=K\conv F$. Hence
\eq\label{eq:PhiK}
&&\Phi_K(F)= K\conv F=\reim{q_1}(K\ltens\opb{q_2}F). 
\eneq
It is natural to consider the right adjoint functor $\Psi_K$ of the functor $\Phi_K$ (see~\cite{KS90}*{Prop.~3.6.2}) given by
\eq\label{eq:PsiK}
&&\Psi_K(G)= \roim{q_2}\rhom(K,\epb{q_1}G).
\eneq
Given three spaces $X_i$ ($i=1,2,3$) and kernels $K_1$ on $X_{12}$  and $K_{2}$ on $X_{23}$,  one has
(by~\eqref{eq:assockernels} or~\cite{KS90}*{Prop.~3.6.4})
\eq
&&\Phi_{K_2}\conv\Phi_{K_1}\simeq\Phi_{K_2\conv K_1},\quad \Psi_{K_1}\conv\Psi_{K_2}\simeq\Psi_{K_2\conv K_1}.
\eneq
\begin{proposition}\label{pro:dualker}
	Let $K\in\Derb(\cor_{X\times X})$ and $F\in\Derb(\cor_X)$. Then 
	$\RD_X(\Phi_K(F))\simeq\Psi_{\oim{v}K}(\RD_XF)$.
\end{proposition}
\begin{proof}
	One has the sequence of isomorphisms
	\eqn
	\RD_X(\Phi_K(F))&\simeq&\rhom(\reim{q_1}(K\ltens\opb{q_2}F),\omega_X)\\
	&\simeq&\roim{q_1}\rhom(K\ltens\opb{q_2}F,\omega_{X\times X})\\
	&\simeq&\roim{q_1}\rhom(K,\rhom(\opb{q_2}F,\epb{q_2}\omega_X))\\
	&\simeq&\roim{q_1}\rhom(K,\epb{q_2}\RD_XF).
	\eneqn
\end{proof}
Also note that when $X_2=\rmpt$, that is, $F,K\in \Derb(\cor_{X})$, then
\eq
&&F\conv K\simeq\rsect_c(X;F\ltens K).
\eneq

\subsubsection*{Non proper composition}

In many situations, the non proper composition is useful. For $K_1\in\Derb(\cor_{X_{12}})$ and $K_2\in\Derb(\cor_{X_{23}})$, one sets
\eq\label{eq:npconv}
&&K_1\npconv K_2=\roim{q_{13}}(\opb{q_{12}}K_1\ltens\opb{q_{23}}K_2).
\eneq
One shall be aware that in general, this composition is not associative.
However, under suitable hypotheses, it becomes associative.

Consider the diagram of good topological spaces
\eq\label{diag:1234}       
&&\ba{l}\xymatrix{
	&& X_{123}\ar[ld]|-{q_{12}}\ar[d]|-{q_{13}}\ar[rd]|-{q_{23}}&&\\
	&X_{12}\ar[dl]|-{q_1}\ar[rd]|-{q_2}&X_{13}\ar[rrd]|-{p_2}\ar[lld]|-{p_1}&X_{23}\ar[ld]|-{r_1}\ar[dr]|-{r_2}&&\\
	X_1&&X_2&&X_3
}\ea\eneq
Note that the squares $(X_{12},X_2,X_{23},X_{123})$, $(X_{12},X_1,X_{13},X_{123})$ and  
$(X_{13},X_3,X_{23},X_{123})$ are Cartesian. 

\begin{lemma}\label{le:assocnp}
	Let $X_i$ \lp$i=1,2,3$\rp\, be three $C^\infty$-manifolds. Let $K_1\in\Derb(\cor_{X_{12}})$ and $K_2\in\Derb(\cor_{X_{23}})$. Assume that
	$K_1$ is cohomologically constructible and 
	$\SSi(K_1)\cap (T^*_{X_1}X_1\times T^*X_2)\subset T^*_{X_{12}}X_{12}$. Then 
	\eqn
	&&\roim{q_{12}}(\opb{q_{12}}K_1\ltens\opb{q_{23}}K_2)\simeq K_1\ltens\roim{q_{12}}\opb{q_{23}}K_2.
	\eneqn
\end{lemma}
\begin{proof}
	Applying~\cite{KS90}*{Prop.~5.4.1}, we have
	\eqn\label{eq:ssK22}
	&&\SSi(\opb{q_{23}}K_2)\subset T^*_{X_1}X_1\times T^*X_{23},\\
	&&\SSi(\epb{q_2}\roim{r_1}K_2)\subset T^*_{X_1}X_1\times T^*X_2.
	\eneqn
	Since $\roim{q_{12}}\epb{q_{23}}K_2\simeq  \epb{q_2}\roim{r_1}K_2$ and 
	$\SSi(\roim{q_{12}}\opb{q_{23}}K_2)=\SSi(\roim{q_{12}}\epb{q_{23}}K_2)$, we get:
	\eq\label{eq:ssK2}
	&&\SSi(\roim{q_{12}}\opb{q_{23}}K_2)\subset T^*_{X_1}X_1\times T^*X_2.
	\eneq
	Applying~\cite{KS90}*{Cor.~6.4.3} we get by the hypothesis and~\eqref{eq:ssK2}
	\eq\label{eq:ssK23}
	K_1\ltens\roim{q_{12}}\opb{q_{23}}K_2&\simeq &\rhom(\RD'_{12}K_1, \roim{q_{12}}\opb{q_{23}}K_2).
	\eneq
	Moreover, the hypothesis implies $\SSi(\RD'_{12}K_1)\cap (T^*_{X_1}X_1\times T^*X_2)\subset T^*_{X_{12}}X_{12}$, hence
	\eqn
	\SSi(\opb{q_{12}}\RD'_{12}K_1)\cap  T^*_{X_1}X_1\times T^*X_{23} &\subset &T^*_{X_{123}}X_{123}.
	\eneqn
	The sheaf $K_1$ being cohomologically constructible on $X_{12}$, the sheaf $\opb{q_{12}}K_1\simeq K_1\etens\cor_{X_3}$ is 
	cohomologically constructible on $X_{123}$. Applying again~\cite{KS90}*{Cor.~6.4.3}, we get
	\eqn
	\rhom(\opb{q_{12}}\RD'_{12}K_1,\opb{q_{23}}K_2) &\simeq&\RD'_{123}\opb{q_{12}}\RD'_{12}K_1\ltens \opb{q_{23}}K_2\\
	&\simeq&\opb{q_{12}}K_1\ltens \opb{q_{23}}K_2.
	\eneqn
	To conclude, note that 
	\eqn
	\rhom(\RD'_{12}K_1, \roim{q_{12}}\opb{q_{23}}K_2).&\simeq &\roim{q_{12}}\rhom(\opb{q_{12}}\RD'_{12}K_1,\opb{q_{23}}K_2)\\
	&\simeq &\roim{q_{12}}(\opb{q_{12}}K_1\ltens\opb{q_{23}}K_2).
	\eneqn
	Using~\eqref{eq:ssK23}, the proof is complete.
\end{proof}

\begin{theorem}\label{th:assocnp}
	Let $X_i$ \lp$i=1,2,3,4$\rp\, be four $C^\infty$-manifolds and let $K_i\in\Derb(\cor_{X_{i,i+1}})$ \lp $i=1,2,3$\rp. 
	Assume that
	$K_1$ is cohomologically constructible,  $q_1$ is proper on $\supp(K_1)$ and 
	$\SSi(K_1)\cap (T^*_{X_1}X_1\times T^*X_2)\subset T^*_{X_{12}}X_{12}$. Then
	\eqn
	&&K_1\npconv[2](K_2\npconv[3]K_3)\simeq (K_1\npconv[2]K_2)\npconv[3] K_3.
	\eneqn
\end{theorem}
\begin{proof}
	We shall assume for simplicity that $X_4=\rmpt$. Consider Diagram~\ref{diag:1234}. Then:
	\eqn
	K_1\npconv[2](K_2\npconv[3] K_3)
	&=&\roim{q_1}\bl K_1\ltens \opb{q_2}(K_2\npconv K_3)\br\\
	&=&\roim{q_1}\bl K_1\ltens \opb{q_2}\roim{r_1}(K_2\ltens\opb{r_2} K_3)\br\\
	&\simeq&\roim{q_1}\bl K_1\ltens\roim{q_{12}}\opb{q_{23}}(K_2\ltens\opb{r_2} K_3)\br\\
	&\simeq&\roim{q_1}\roim{q_{12}}\bl\opb{q_{12}}K_1\ltens \opb{q_{23}}K_2\ltens\opb{q_{23}}\opb{r_2} K_3\br\\
	&\simeq&\roim{p_1}\roim{q_{13}}\bl{}\opb{q_{12}}K_1\ltens \opb{q_{23}}K_2\ltens\opb{q_{13}}\opb{p_2} K_3\br\\
	&\simeq&\roim{p_1}\lp\reim{q_{13}}(\opb{q_{12}}K_1\ltens \opb{q_{23}}K_2)\ltens \opb{p_2} K_3\br\\
	&\simeq&\roim{p_1}\bl(K_1\cconv[2] K_2)\ltens\opb{p_2} K_3\br\simeq \roim{p_1}\bl(K_1\npconv[2] K_2)\ltens\opb{p_2} K_3\br.
	\eneqn
	In the first isomorphism, we have used   $\opb{q_2}\roim{r_1}\simeq \roim{q_{12}}\opb{q_{23}}$, which follows from  the isomorphism  $\epb{q_2}\roim{r_1}\simeq \roim{q_{12}}\epb{q_{23}}$.
	In the second isomorphism, we have used Lemma~\ref{le:assocnp}.
	In the fourth  isomorphism, we have used the fact that $q_{13}$ is proper on $\supp(\opb{q_{12}}K_1)$. Finally, in the sixth isomorphism we have again used  the fact that $q_{13}$ is proper on $\supp\opb{q_{12}}(K_1)$. 
	
	Note that the same proof holds without assuming $X_4=\rmpt$. In this case replace $X_i, X_{ij}$ and $X_{123}$ with 
	$X_{i4}, X_{ij4}$ and $X_{1234}$, respectively.
\end{proof}

\subsection{Monoidal presheaves}\label{sect:monoid}

We shall use the theory of monoidal categories and refer to~\cites{Kas95} and~\cite{KS06}*{Ch.~IV}. Note that 
\begin{itemize}
	\item
	monoidal categories are called tensor categories in~\cite{KS06},
	\item
	to a monoidal category $(\shc,\otimes)$ is naturally attached an isomorphism of functors  (\cite{KS06}*{Def.~4.2.1})
	$\asso(X,Y,Z)\cl(X\otimes Y)\otimes Z\isoto X\tens(Y\otimes Z)$ satisfying the usual compatibility conditions,
	\item
	to a monoidal category with unit  $(\shc,\otimes,\un)$ are naturally attached two functorial isomorphisms
	$\run\cl X\otimes \un\to X$ and $\lun\cl \un\otimes X\to X$, denoted respectively $\alpha$ and $\beta$ in~\cite{KS06}*{Lem.~4.2.6}.
\end{itemize} 
\begin{example}
	(i) We regard the ordered set $(\R, \leq)$ as a category  that we simply denote by $\R$ and we regard $\R_{\geq0}$ as a full subcategory.
	The categories $\R$ and $\R_{\geq0}$ endowed with the addition map $+$ are monoidal categories with unit, denoted 
	$(\R,+)$ and  $(\R_{\geq0},+)$, respectively.
	
	\spa
	(ii) Let $X$ be a good topological space. The category $(\Derb(\cor_{X\times X},\conv))$ is a monoidal category with unit the sheaf $\cor_\Delta$. 
	
	\spa
	(iii) If $\sha$ is a category, then the category $(\Fct(\sha, \sha),\conv)$ is a monoidal category with unit the object $\id_\sha$. 
\end{example}

Let $I$ be a closed  interval of $\R$. We assume
\eq\label{hyp:I}
&&\parbox{70ex}{either $I=[0,\alpha]$ or $I=[-\alpha,\alpha]$ for some $\alpha>0$.
}\eneq
We consider $I$ as an ordered set and we denote by $I_\leq$ the associated category,  a full subcategory  of $(\R,\leq )$. 
Hence, $\Ob(I_\leq)=I$ and $\Hom[I_\leq](a,b)=\rmpt$ or $=\varnothing$ according whether $a\leq b$ or not.
Although it has not been precisely defined,  we shall look  at $I_\leq$ as a ``partially monoidal subcategory of $(\R,+)$''.

Let $(\shc,\otimes)$ be  a monoidal  category and consider a presheaf $K$ on $I_\leq$ with values in $\shc$. For $a\in I$, we write $K_a$ instead of $K(a)$. Hence, we have ``restriction'' morphisms $\rho_{a,b}\cl K_b\to K_a$ for $a,b\in I, a\leq b$ satisfying the usual compatibility relations $\rho_{a,b}\circ \rho_{b,c}=\rho_{a,c}$ for $a\leq b\leq c$ and $\rho_{a,a}=\id$.

\begin{definition}\label{def:pshK}
	Let $(\shc,\otimes,\un)$ be a monoidal category with unit.
	\banum
	\item
	A monoidal presheaf $(K,\phi_0,\phi_2)$ on $I_\leq$ with values in $\shc$ is the data of :
	
	\spa
	(1) a presheaf $K$ on $I_\leq$ with value in $\shc$,
	
	\spa
	(2) an isomorphism $ \phi_0 \cl {\bf 1} \isoto K_0$,
	
	\spa
	(3) an isomorphism $ \phi_2(a,b) \colon K_a\otimes K_b\isoto K_{a + b}$, for $a,b$ such that $a,b,a+b\in I$, 
	
	these data satisfy the following conditions:
	
	\spa
	(i) the diagram below commutes for all   $a,b,a',b'\in I$ such that $a\leq a'$, $b\leq b'$, $a,b,a',b', a+b,a'+b'\in I$:
	\eqn
	&&\xymatrix@C=1,5cm{
		K_{a'}\otimes K_{b'}\ar[d]_-{\rho_{a, a^\prime} \otimes \rho_{b, b^\prime}}\ar[r]_\sim^-{\phi_2(a^\prime,b^\prime)} &K_{a'+b'}\ar[d]^-{\rho_{a+b,a^\prime+b^\prime}}\\
		K_{a}\otimes K_{b}\ar[r]_\sim^-{\phi_2(a,b)}           &K_{a+b}.
	}
	\eneqn
	Here, the vertical arrows are induced by the restriction morphisms.
	
	\spa
	(ii) For all $a, b, c \in I$ such that $a+b, b+c, a+b+c \in I$, the diagram below commutes
	\eqn
	&&\xymatrix@C=1cm{
		(K_a \otimes K_b) \otimes K_c \ar[rr]^-{\asso({K_a , K_b, K_c})} \ar[d]_-{\phi_2(a,b) \otimes \id}  && K_a \otimes( K_b \otimes K_c) \ar[d]^-{\id \otimes \phi_2(b,c)} \\
		K_{a+b} \otimes K_c \ar[d]_-{\phi_2(a+b,c)} && K_a \otimes K_{b+c} \ar[d]^-{\phi_2(a,b+c)}\\
		K_{a+b+c}  \ar@{=}[rr] &&  K_{a+b+c}\, .
	}\eneqn
	
	\spa
	(iii) For all $a \in I$, the diagrams below commute
	\eqn
	&&\xymatrix@C=1cm{
		\un \otimes K_a \ar[r]^-{\lun_{K_a}} \ar[d]_-{\phi_0 \otimes \id_{K_a}}  & K_a &&  K_a \otimes \un  \ar[r]^-{\run_{K_a}} \ar[d]_-{ \id_{K_a} \otimes \phi_0} & K_a\\
		K_0 \otimes K_a \ar[r]^-{\phi_2(0,a)}& K_a, \ar@{=}[u]               && K_a \otimes K_0 \ar[r]^-{\phi_2(a,0)}& K_a.\ar@{=}[u]
	}
	\eneqn
	\item
	Let $K$ and $K^\prime$ be two monoidal presheaves on $I_\leq$.  A morphism of  monoidal  presheaves $\eta \colon K \to K^\prime$ is a morphism such that for every $a, b \in I$ such that $a+b \in I$ the following diagram commutes
	\eqn
	&&\xymatrix@C=1,5cm{
		K_a \otimes K_b \ar[r]^-{\eta_a \otimes \eta_b} \ar[d]_-{\phi_2(a,b)} & K^\prime_a \otimes K^\prime_b \ar[d]^-{\phi^\prime_2(a,b)}\\
		K_{a+b} \ar[r]^-{\eta_{a+b}} & K^\prime_{a+b}. 
	}\eneqn
	\item
	We denote by $\Fun^\otimes(I^\op,\shc)$ the category whose objects are the monoidal presheaves on $I_\leq$ with values in $\shc$ and the morphisms are the morphisms of  monoidal  presheaves.
	\eanum
\end{definition}
Assuming that $I=[0,\alpha]$, 
the inclusion functor $i_{\alpha} \colon I_\leq\into\R_{\geq 0}$ induces a functor
\eq\label{eq:ialpha1}
\spb{i_\alpha} \colon \Fun^{\otimes}(\R_{\geq 0}^\op,\shc) \to \Fun^{\otimes}(I^\op,\shc), \; F \mapsto F \conv i_\alpha.
\eneq
Similarly, if  $I=[-\alpha,\alpha]$,  the inclusion functor $j_{\alpha} \colon I_\leq\into\R_{\geq 0}$ induces a functor
\eq\label{eq:ialpha2}
\spb{j_{\alpha}} \colon \Fun^{\otimes}(\R^\op,\shc) \to \Fun^{\otimes}(I^\op,\shc), \; F \mapsto F \conv j_\alpha.
\eneq

\begin{theorem}\label{th:pshK}
	Assuming that $I=[0,\alpha]$,
	the functor $\spb{i_\alpha}$ in~\eqref{eq:ialpha1} is an equivalence of categories. 
	Similarly, assuming that $I=[-\alpha,\alpha]$,
	the functor $\spb{j_\alpha}$ in~\eqref{eq:ialpha2} is an equivalence of categories. 
\end{theorem}
\begin{proof}
	(A) Let us first treat the case $I=[0,\alpha]$.
	
	\spa
	It follows from \cite{Kas95}*{Ch XI.5} that we can assume that $\shc$ is a strict monoidal category. We set $\lambda=\frac{\alpha}{2}$.
	
	\spa
	(i) We start by showing that $\spb{i_\alpha}$ is essentially surjective. For that purpose, given a monoidal presheaf $K$ on $I$, 
	we will construct a monoidal presheaf $\stk \colon \R_{\geq 0} \to \shc$  such that $\spb{i_\alpha} \stk \simeq K$.
	
	\spa 
	(i)--(a)  For $a\geq0$ we write $a= n {\lambda}+r_a$ with $0\leq r_a<\lambda$. Then, one sets 
	\eq\label{def:stka}
	&&\stk_a \eqdot \underbrace{K_{\lambda}\otimes\cdots\otimes K_{\lambda}}_n\otimes\ K_{r_a}.
	\eneq
	
	\spa
	(i)--(b)  We now construct the restriction morphisms $\rho_{a,b}$.
	For $a\leq b \leq \lambda$,  $\rho_{a,b}$ is given by the definition of the presheaf $K$.
	Let us write  $a= m \cdot \lambda +r_a$ and $b=n\cdot  \lambda +r_b$ with $0 \leq  r_a , r_b<\lambda$.
	Since $0 \leq a \leq b$, $m\leq n$. If $m=n$, then $r_a\leq r_b$ and we set 
	$\rho_{a,b}\eqdot  ( \id_{K_\lambda})^{\conv m} \conv \rho_{r_a,r_b}$.
	
	\spa
	Now assume  $m>n$. Notice that 
	\eqn
	&&\stk_b \simeq (K_{\lambda})^{\conv m} \conv K_{\lambda} \conv  (K_{\lambda})^{\conv (n-m-1)} \conv K_{r_b}\\
	&&\stk_a \simeq (K_{\lambda})^{\conv m} \conv K_{r_a} \conv (K_0)^{\conv (n-m-1)} \conv K_0.
	\eneqn
	Hence, we set $\rho_{a,b}\eqdot ( \id_{K_\lambda})^{\conv m} \conv \rho_{r_a,\lambda} \conv (\rho_{0,\lambda})^{\conv n-m-1} \conv \rho_{0,r_b}$.
	
	\spa
	(i)--(c) Let us construct the isomorphisms $ \phi_2(a_1,a_2) \colon \stk_{a_1}\otimes \stk_{a_2} \to \stk_{a_1+a_2}$, for $a_1,a_2 \in \R_{\geq 0}$. Write
	\eqn
	&&a_i=n_i \cdot \alpha +r_i,\quad 0 \leq r_i < \lambda,\quad i=1, 2.
	\eneqn
	Since $r_{i} + \lambda \leq \alpha$, \; $K_{r_{i}}\otimes K_\lambda \stackrel{\phi_2(r_i,\lambda)}{\simeq} K_{r_{i} + \lambda}  \stackrel{\phi^{-1}_2(\lambda,r_i)}{\simeq} K_\lambda \otimes K_{r_{i}}$. We set 
	\eqn
	s_i\eqdot \phi^{-1}_2(\lambda,r_i) \circ \phi_2(r_i,\lambda)
	\eneqn
	
	Let $n \in \N$ and consider the map
	\eqn
	\psi_{i,n} \eqdot (\id_{K_\lambda}^{\otimes n-1} \otimes s_i) \circ \ldots \circ (\id_{K_\lambda}^{\otimes p} \otimes s_i \otimes \id_{K_\lambda}^{\otimes n-1-p}) \circ \ldots \circ (s_i \otimes \id_{K_\lambda}^{\otimes n-1}).
	\eneqn
	
	We now define the map $ \phi_2(a_1,a_2) \colon \stk_{a_1} \otimes \stk_{a_2} \to \stk_{a_1+a_2}$ by setting 
	\eqn
	\phi_2(a_1,a_2)  \eqdot (\id_{K_{\lambda}^{\otimes (n_1+n_2)}} \otimes \phi_2(r_1,r_2))\circ(\id_{K_{\lambda}}^{\otimes n_1} \otimes \psi_{1,n_2} \otimes \id_{K_{r_2}}). 
	\eneqn
	By construction, $\phi_2(a_1,a_2) $ is an isomorphism.

	It is straightforward  to check that  $\stk$ is a monoidal presheaf on $\R_{\geq 0}$ and that  $\spb{i_\alpha}\stk \simeq K$.
	
	\spa
	(ii)-(a) Let us prove that  $\spb{i_\alpha}$ is faithful. Let $f,g \colon \stk \to \stk^\prime$ be two monoidal morphisms between monoidal presheaves on $\R_{\geq 0}$. Assume that $\spb{i_\alpha}(f)=\spb{i_\alpha}(g)$. Hence, for every $0 \leq a \leq \alpha$, $f_a=g_a$ and it follows from the definition of a monoidal morphism that for every $b \in \R_{\geq 0}$, $f_b=g_b$.  
	
	\spa
	(ii)-(b) Let us show that $\spb{i_\alpha}$ is full.
	Let $\stk, \stk^\prime \in \Fun^\otimes(\R_{\geq 0}^\op, \shc)$ and let $f\cl \spb{i_\alpha}\stk \to \spb{i_\alpha}\stk^\prime$ be a monoidal morphism. For $a \in \R_{\geq 0}$,  we write $a= n {\lambda}+r_a$ with $0\leq r_a<\lambda$. We define the morphism $\stf_a$ as the composition
	\eqn
	&&\xymatrix{
		\stk_a \simeq \stk_\lambda^{\otimes n} \otimes \stk_{r_a}\ar[rr]^{f_{\lambda}^{\otimes n} \otimes f_{r_a}}&&\stk_\lambda^{\prime \otimes n} \otimes \stk^\prime_{r_a} \simeq \stk^\prime_a.
	}\eneqn
	The family of morphisms $(\stf_a)_{a \in \R_{\geq 0}}$ defines a monoidal morphism $\stf:\stk \to \stk^\prime$ such that $i_\alpha(\stf)=f$.
	
	\spa
	(B) Assume now that $I=[-\alpha,\alpha]$.
	Part (A) of the proof applies when replacing the interval $[0,\alpha]$ and $\R_{\geq0}$ with 
	the interval $[-\alpha,0]$ and $\R_{\leq0}$. Then combine these two cases. 
\end{proof}

\subsection{Thickening kernels and interleaving distance}

Let us first recall that  a categorical axiomatic for interleaving distances was developed in~\cites{Bub14, SMS18}. 
Here, we do not work in an abstract categorical setting but restrict ourselves  to the study of kernels for sheaves,
a natural framework for applications.

\begin{definition}\label{def:thickdiag}
	Let $X$ be a good topological space.
	\banum
	\item
	A thickening kernel  is a  monoidal presheaf $\stK$ on $(\R_{\geq0},+)$ 
	with values in the monoidal category 
	$(\Derb(\cor_{X\times X}),\conv)$. 
	\item
	The thickening kernel $\stK$  is a bi-thickening kernel if it extends as a monoidal presheaf on $(\R,+)$.
	\eanum
\end{definition}
In the sequel, for a thickening (resp.\ a bi-thickening) kernel $\stK$, one sets $\stK_a=\stK(a)$ for $a\geq0$
(resp.\  for $a\in\R$). 

In other words,
a thickening kernel is a family of kernels $\stK_a\in\Derb(\cor_{X\times X})$ satisfying 
\eqn
&&\stK_a\conv\stK_b\simeq\stK_{a+b},\quad \stK_0\simeq\cor_{\Delta}\mbox{ for }a\in\R_{\geq0}
\eneqn
and the compatibility conditions of Definition~\ref{def:pshK}. 

We shall often simply write  ``a thickening'' instead of   ``a thickening kernel ''. 
\begin{remark}\label{rem:unithick}
	Let $I=[0,\alpha]$. Note that if the thickening (or the bi-thickening) $\stK$ exists, then it is uniquely defined by its restriction to $[0,\alpha]$, 
	up to isomorphism. This last isomorphism is unique in the following sense.
	
	Denote by  $K_I$  the monoidal presheaf 
	$a\mapsto\cor_{\Delta_a}$ on $I_\leq$. 
	Given two  thickenings $\stK_1$ and $\stK_2$ and isomorphisms of monoidal presheaves 
	\eqn
	&&\theta\cl \stK_1\vert_I\isoto K_I,\quad \theta'\cl \stK_2\vert_I\isoto K_I,
	\eneqn
	then there exists a unique isomorphism of monoidal presheaves $\lambda\cl\stK_1\isoto\stK_2$ such that 
	$\lambda\vert_I=\opb{\theta'}\circ\theta$. (Here we use the notation $\scbul\vert_I$ instead of $\spb{i_\alpha}$ as in~\eqref{eq:ialpha1}.)
\end{remark}

\begin{example}\label{exa:monoidpsh}
	(i) The constant presheaf $a\mapsto \cor_\Delta$ is a thickening kernel  called the constant thickening on $X$ and simply denoted $\cor_{\Delta}$ (or $\cor_{\Delta_X}$ if necessary).
	
	\spa
	(ii) Let $X_i$ ($i=1,2$) be two good topological spaces and let $\stK_i$  be a thickening kernel  on $X_i$. 
	Then 
	$\stK_1\etens\stK_2$ is a  thickening kernel  on $X_1\times X_2$. This applies in particular when $\stK_i$ is the constant thickening  on $X_1$ or $X_2$.
	
	\spa
	(iii) Let $(X,d_X)$ be a  metric space. We shall prove in Theorem~\ref{th:thick1} below that, under suitable hypotheses, there exists a thickening kernel 
	$\stK$ with $\stK_a=\cor_{\Delta_a}$ for $0\leq a\leq\alpha_X$.
	For $S$  a good topological space, 
	we  sometimes denote by $\stK_{S\times X/S}$ the thickening kernel  $\cor_{\Delta_S}\etens\stK$.
	
	\spa
	(iv) Another example of a thickening kernel   will be given in \S~\ref{subsect:negatisotop} in which we use the kernel of~\cite{GKS12} associated with a  Hamiltonian isotopy.
\end{example}

The next definition is mimicking~\cite{KS18}*{Def.~2.2}.

\begin{definition}\label{def:convdist}
	Let $\stk$ be a thickening kernel on $X$,
	let $F, G \in \Derb(\cor_X)$ and let $a\geq0$.
	\banum
	\item
	One says that $F$ and $G$ are {\em $a$-isomorphic} if there are morphisms 
	$f\cl \stK_a\conv F\to G$ and $g\cl \stK_a\conv G\to F$ 
	which satisfy the following compatibility conditions:
	the composition 
	\begin{equation*}
		\stK_{2a}\conv F\To[\stK_{a} \conv f] \stK_a\conv G\To[g] F
	\end{equation*}
	and the composition 
	\begin{equation*}
		\stK_{2a}\conv G\To[\stK_{a} \conv g] \stK_a\conv F\To[f] G
	\end{equation*}
	coincide with the morphisms induced by the canonical morphism $\rho_{0,2a} \colon \stk_{2a} \to \stk_0$.
	\item
	One sets
	\begin{equation*}
		\dist_\stK(F,G)=\inf\Bigl(\{+\infty\}\cup\{a\in\R_{\geq0}\,;\,
		\text{$F$ and $G$ are $a$-isomorphic}\}\Bigr)
	\end{equation*}
	and calls $\dist_\stK(\scbul,\scbul)$ the interleaving distance (associated with $\stK$).
	\eanum
\end{definition}
Note that if $F$ and $G$ are $a$-isomorphic, then
they are $b$-isomorphic for any $b\ge a$. 

The next result show that the interleaving distance $\dist_\stK$ is a pseudo-distance on  $\Derb(\cor_X)$.
\begin{proposition}\label{pro:convdist}
	Let  $\stK$  be  a thickening kernel on $X$ and let 
	$F,G,H\in \Derb(\cor_X)$. Then 
	\bnum
	\item $F$ and $G$ are $0$-isomorphic if and only if $F\simeq G$,
	\item
	$\dist_\stK(F,G)=\dist_\stK(G,F)$, 
	\item
	$\dist_\stK(F,G)\leq \dist_\stK(F,H)+\dist_\stK(H,G)$.
	\enum
\end{proposition}
The proof is straightforward.

\begin{remark}
	It is proved in \cite{PSW21} that if $X_\infty$ is a b-analytic manifold (see \cite{Sc20})  endowed with a good distance, the pseudo-distance 
	$\dist_\stK$ becomes a distance when restricted to the category $\Derb_{\Rc}(\cor_{X_\infty})$ of  sheaves constructible up to infinity. In particular, on any real analytic manifold $X$,  $\dist_\stK$ becomes a distance when restricted to constructible sheaves with compact support. 
	Let us also mention the paper~\cite{CR19} in which it is shown that the category  $\Derb_{\Rc}(\cor_{X})$ is not metrically complete.
\end{remark}

\section{The interleaving distance on metric spaces}\label{sect:metric}
From now on and until the end of this section, unless otherwise stated, we assume that $X$ is a good topological space.

\subsection{Thickening of the diagonal} 

Let $(X,d_X)$ be a metric space. For $a\geq0$, $x_0\in X$, set
\eq\label{eq:BDG}
&&\left\{ \ba{l}B_a(x_0)=\{x\in X;d_X(x_0,x)\leq a\},\\
B^\circ_a(x_0)=\{x\in X;d_X(x_0,x)< a\},\, (\mbox{here, $a>0$}),\\
\Delta_a=\{(x_1,x_2)\in X\times X;d_X(x_1,x_2)\leq a\},\\
\Delta^\circ_a=\{(x_1,x_2)\in X\times X;d_X(x_1,x_2)<a\},\, (\mbox{here, $a>0$}),\\
Z=\{(x_1,x_2,t)\in X\times X\times\R_{\geq0};d_X(x_1,x_2)\leq t, \, t<\alpha_X\},\\
\Omega^+=\{(x_1,x_2,t)\in X\times X\times\R_{>0};d_X(x_1,x_2)< t,\, t<\alpha_X\}.
\ea\right.\eneq
\begin{definition}\label{def:goodmetricsp}
	A metric space $(X,d_X)$ is good if the underlying topological space is good and moreover there exists some  $\alpha_X>0$
	such that for all $0\leq a,b$ with $a+b\leq\alpha_X$, one has
	\eq\label{hyp:dist1}
	&&\left\{\parbox{65ex}{
		(i) for any $x_1,x_2\in X$, $B_a(x_1)\cap B_b(x_2)$ is contractible or empty (in particular, for any $x\in X$, $B_a(x)$ is contractible),\\
		(ii) the two projections $q_1$ and $q_2$ are proper on $\Delta_a$,\\
		(iii) $\Delta_a\conv\Delta_b=\Delta_{a+b}$.
	}\right.
	\eneq
\end{definition}
Clearly, in this definition, $\alpha_X$ is not unique.
In the sequel, if we want to  mention which $\alpha_X$ we choose, we denote the good metric space by 
$(X,d_X,\alpha_X)$.

Let $U$ be an open  subset of a real $C^0$-manifold $M$. Recall (see~\cite{KS90}*{Exe.~III.4}) that $U$ is 
locally cohomologically trivial (l.c.t. for short) in $M$ if for each $x\in \ol{U}\setminus U$, $(\rsect_{\ol U}(\cor_M))_x\simeq0$ and 
$(\rsect_{U}(\cor_M))_x\simeq\cor$. 

We shall say that  $U$  is locally topologically convex  (l.t.c. for short)  in $M$ if each $x\in M$ admits an open neighborhood $W$ such that there exists a topological isomorphism $\phi\cl W\isoto V$, with $V$ open in a real vector space, such that $\phi(W\cap U)$ 
is convex.
Clearly, if $U$ is  l.t.c. then it is l.c.t.

Moreover,  the natural morphism  
$\cor_U\to \cor_M$ defines a section of $\Hom(\cor_U,\cor_M)\simeq\Hom(\cor_U\tens\cor_{\ol U},\cor_M)$, hence defines the morphisms:
\eqn
&&\cor_U\to\RD'_M\cor_{\ol U},\quad \cor_{\ol U}\to\RD'_M\cor_{U}.
\eneqn										
When $U$ is  l.c.t., then these morphisms are isomorphisms.
If moreover,  $U$ is  l.t.c.,  then these sheaves are cohomologically constructible.

We shall also encounter the hypotheses:
\eq\label{hyp:dist2}
&&\left\{\parbox{65ex}{
	The good metric  space  $X$ is a $C^0$-manifold and 
	\banum
	\item
	for   $0<a\leq\alpha_X$, the set $\Delta_a^\circ$ is l.t.c. in $X\times X$,
	\item
	the set $\Omega^+$ is l.t.c.  in $X\times X\times]-\infty,\alpha_X[$.
	\item
	For  $x,y\in X$, setting $Z_{a}(x,y)=B_a(x)\cap B^\circ_a(y)$, one has $\rsect(X;\cor_{Z_{a}(x,y)})\simeq0$ for $x\neq y$ and  $0<a\leq\alpha_X$. 
	\eanum
}\right.
\eneq
%
Therefore
\begin{lemma}\label{lem:add2}
	Let $(X,d_X)$ be a good metric space  satisfying~\eqref{hyp:dist2}.
	\banum
	\item 
	The sheaves $\cor_{\Delta_a}$ and $\cor_{\Delta^\circ_a}$  are 
	cohomologically constructible and dual one to each other  for the duality functor $\RD'_{X\times X}$.
	\item 
	The sheaves $\cor_Z$ and $\cor_{\Omega^+}$ are 
	cohomologically constructible and  dual one to each other  for the duality functor $\RD'_{X\times X\times\R}$.
	\eanum
\end{lemma}

The next hypothesis will be used in order to apply Theorem~\ref{th:assocnp} and we shall give in Lemma~\ref{lem:add5}  below a natural criterion  in order that it is satisfied. 

\eq\label{hyp:dist4}
&&\left\{\parbox{65ex}{
	The good metric  space  $X$ is a $C^\infty$-manifold and, for  $0<a\leq\alpha_X$,
	$\SSi(\cor_{\Delta_a})\cap (T^*_XX\times T^*X)\subset T^*_{X\times X}X\times X$.
}\right.
\eneq

\begin{lemma}\label{lem:add5}
	Let $(X,d_X)$ be a good metric space. Assume that $X$ is a  $C^\infty$-manifold,
	the distance function $f\eqdot d_X\cl X\times X\to\R$ is of class $C^1$ on $W\eqdot \Delta^\circ_a\setminus\Delta$ for $a\leq\alpha_X$ and the partial differentials $d_xf$ and $d_yf$ do not vanish on $W$.
	Then~\eqref{hyp:dist4} is satisfied.
\end{lemma}
\begin{proof}
	Apply \cite{KS90}*{Prop.~5.3.3}.
\end{proof}

We shall obtain in Theorems~\ref{th:normedtspace} and~\ref{th:riemann1}  large classes of examples in which hypotheses~\eqref{hyp:dist1},~\eqref{hyp:dist2} and~\eqref{hyp:dist4} are satisfied.

\begin{lemma}\label{lem:add1}
	Let $(X,d_X)$ be a good metric space.
	\banum
	\item For every $a,b \geq 0$, $\cor_{\Delta_a} \conv \cor_{\Delta_b}  \simeq \cor_{\Delta_b} \conv \cor_{\Delta_a}$.
	\item For any $0\leq a,b$ with $a+b\leq\alpha_X$,
	\eq\label{e	q:dist1}
	&&\cor_{\Delta_a}\conv\cor_{\Delta_b}\simeq\cor_{\Delta_{a+b}},
	\eneq
	and the correspondence $a\mapsto \cor_{\Delta_a}$ defines a monoidal presheaf on $[0,\alpha_X]$ with values in the monoidal category $(\Derb(\cor_{X\times X}),\conv)$.
	\eanum
\end{lemma}
\begin{proof}
	(a) Recall Notations~\eqref{eq:volte}.  
	Since  $\opb{v} \cor_{\Delta_a} \simeq \cor_{\opb{v}(\Delta_a)} \simeq \cor_{\Delta_a}$, the result follows.
	
	\spa
	(b) We shall follow the notations of~\eqref{diag:123} (with $X_i=X$ for all $i$). Setting $\Delta_a\times_2\Delta_b=\opb{q_{12}}\Delta_a\cap\opb{q_{23}}\Delta_b$,  we have
	\eqn
	&&\opb{q_{12}}\cor_{\Delta_a}\ltens \opb{q_{23}}\cor_{\Delta_b}\simeq\cor_{\Delta_a\times_2\Delta_b}.
	\eneqn
	The map $q_{13}\cl \Delta_a\times_2\Delta_b\to\Delta_{a+b}$ is proper, surjective and has  contractible fibers by Hypothesis~\eqref{hyp:dist1}. Therefore, $\reim{q_{13}}\cor_{\Delta_a\times_2\Delta_b}\simeq\cor_{\Delta_{a+b}}$ by Lemma~\ref{le:kAcB}. 
	The other conditions in Definition~\ref{def:pshK} are easily checked. 
\end{proof}

We shall refine  Definition~\ref{def:thickdiag}.
\begin{definition}\label{def:metricthick}
	Let $(X,d_X,\alpha_X)$ be a good metric space.
	\banum
	\item
	A {\em metric thickening  kernel of the diagonal}  is a thickening  kernel
	whose restriction to $[0,\alpha_X]$ is isomorphic to the monoidal presheaf 
	$a\mapsto\cor_{\Delta_a}$ on $[0,\alpha_X]$.
	\item
	A metric bi-thickening kernel is a bi-thickening kernel whose restriction to $\R_{\geq0}$ is a metric thickening kernel.
	\eanum
\end{definition}
When there is no risk of confusion, (that is, almost always) we shall simply call  a metric thickening kernel,  ``a thickening''. 

Note that if the metric thickening (or bi-thickening) exists, then it is unique up to isomorphism. This last isomorphism is unique in the  sense of Remark~\ref{rem:unithick}.

\begin{theorem}\label{th:thick1}
	Let $(X,d_X,\alpha_X)$ be a good metric space. 
	There exists a metric thickening $\stD$ of the diagonal.
	Moreover, for each $a\geq0$, the two projections $q_1,q_2\cl X\times X\to X$ are proper on $\supp\stk_a$.
\end{theorem}
\begin{proof}
	The first part of the statement follows from  Lemma~\ref{lem:add1} and Theorem~\ref{th:pshK}. The properness of $q_1$ and $q_2$ on $\supp\stk_a$ for $0\leq a\leq\alpha$ follows from Hypothesis~\eqref{hyp:dist1}. The general case follows from the construction of the kernel.
\end{proof}
\begin{corollary}
	In the preceding situation,
	let $Y$ be a good topological space and let $L\in\Derb(\cor_{X\times Y})$. Then 
	\eqn
	&&\stD_a\conv L\isoto \stD_a\npconv L \mbox{ for }a\geq0.
	\eneqn
\end{corollary}

\subsubsection*{Non proper composition for the distance kernels}

\begin{proposition}\label{pro:add4}
	Let $(X,d_X,\alpha_X)$ be a good metric space  satisfying~\eqref{hyp:dist2} and~\eqref{hyp:dist4}. Then for $a\geq0$, 
	and for smooth real manifolds $X_i$ \lp$i=2,3$\rp\, setting $X=X_1$, we have for any 
	$L_i\in\Derb(\cor_{X_{ij}})$ with $i=1,2$, $j=i+1$,
	\eqn
	&&\stD_a\conv (L_1\npconv L_2)\simeq (\stD_a\conv L_1)\npconv L_2.
	\eneqn
\end{proposition}
\begin{proof}
	(i) Assume first that $0\leq a<\alpha_X$. 
	In this case, $\stD_a=\cor_{\Delta_a}$ is cohomologically constructible and $q_1$ is proper on its support. Using hypothesis~\eqref{hyp:dist4}, we may 
	apply Theorem~\ref{th:assocnp}.
	
	\spa
	(ii)  Assume that the result has been proved for $\stK_b$ (for any kernels $L_1$ and $L_2$) for some $b\geq 0$ and let us prove that it is true for $\stD_{b+a}$ as soon as $0\leq a<\alpha_X$. We have
	\eqn
	\stD_{b+a}\conv (L_1\npconv L_2)&\simeq& \stD_b\conv(\stD_{a}\conv (L_1\npconv L_2))
	\simeq\stD_b\conv((\stD_{a}\conv L_1)\npconv L_2)\\
	& \simeq&(\stD_b\conv(\stD_{a}\conv L_1))\npconv L_2
	\simeq(\stD_{a+b}\conv L_1)\npconv L_2
	\eneqn
\end{proof}

\subsubsection*{Thickening and convolution}

In~\cite{KS18}, the space $X$ is the Euclidian space  $\R^n$ and  the composition $\cor_{\Delta_{a}} \conv$ is replaced by the 
convolution $\cor_{B_a} \star$ where $B_a$ is the closed ball of center $0$.  One can proceed similarly if  the good metric space $(X,d_X)$ is a topological group.
\begin{definition}
	A good metric group $(X,d_X,m,e)$, or simply $(X,d_X)$ for short, is a good metric space $(X,d_X)$ which is  a topological group for the topology induced by the distance, with multiplication $m$ and neutral element $e$, and such that the distance is bi-invariant. In other words, 
	\eqn
	&&d_X(x_1,x_2)=d_X(x_1x_3,x_2x_3)=d_X(x_3x_1,x_3x_2)\mbox{ for }x_1,x_2,x_3\in X.
	\eneqn
\end{definition}
One defines the convolution of  $F, G \in \Derb(\cor_X)$ by 
\eqn
F \star G := \reim{m}(F \etens G).
\eneqn

\begin{proposition}\label{pro:conv}
	Assume that $X$ is a good metric group. Let $B_a$ be the closed ball of radius $a$ centered at the unit $e$. There is a canonical isomorphism of functor
	\eqn
	\cor_{\Delta_a}\conv\simeq \cor_{B_a}\star.
	\eneqn
\end{proposition}
\begin{proof}
	Consider the map  $v\cl X\times X\to X\times X, (x_1,x_2) \mapsto (x_1\opb{x_2},x_2)$. One has 
	$\Delta_a=\opb{v}\opb{q_1}(B_a)$, $\opb{v}\conv\opb{q_2}\simeq\opb{q_2}$ and $m\circ v=q_1$. Therefore,  for $F\in\Derb(\cor_X)$,
	\eqn
	\cor_{B_a} \star F &=& \reim{m}(\cor_{B_a} \etens F)\\
	&\simeq& \reim{m}\reim{v}(\opb{v}\opb{q_1} \cor_{B_a} \otimes \opb{q_2} F)\simeq \cor_{\Delta_{a}} \conv F.
	\eneqn
	We have used $\reim{v}(\opb{v}\opb{q_1} \cor_{B_a} \otimes \opb{q_2} F)\simeq \opb{q_1} \cor_{B_a} \otimes \reim{v}\opb{q_2} F
	\simeq \cor_{B_a} \etens \opb{q_2} F$ which follows from $\eim{v}\circ \opb{v}\simeq\id$.
\end{proof}

\subsection{Bi-thickening of the diagonal}
In this subsection,  $(X,d_X,\alpha_X)$ is a good metric space satisfying~\eqref{hyp:dist2}. When necessary, we denote by $X_i$ ($i=1,2,\dots$) various copies of $X$.

For $a\geq0$, we define the functors $\stL_a$ and $\stR_a$ by
\eq\label{eq:LaRa}
&&\stL_a=\Phi_{\stD_a}=\stD_a\conv=\reim{q_1}\bl\stD_a\ltens\opb{q_2}(\scbul)\br,\quad \stR_a=\Psi_{\stD_a}=
\roim{q_2}\rhom\bl \stD_a,\epb{q_1}(\scbul)\br.
\eneq
Recall that the functor $\stR_a$ is right adjoint to the functor $\stL_a$ (see \cite{KS90}*{Proposition 3.6.2}). 

\begin{lemma}\label{lem:add3}
	Let $(X,d_X,\alpha_X)$ be a good metric space  satisfying~\eqref{hyp:dist2}.
	For $0< a\leq \alpha_X$, 
	$\cor_{\Delta_a}\conv(\cor_{\Delta^\circ_a}\tens\opb{q_2}\omega_X)\simeq\cor_{\Delta}$.
\end{lemma}
\begin{proof}
	Set $S_a=\opb{q_{12}}\Delta_a\cap\opb{q_{23}}\Delta^\circ_a$.
	We have
	\eqn
	\cor_{\Delta_a}\conv\cor_{\Delta^\circ_a}&=&\reim{q_{13}}\bl\opb{q_{12}}\cor_{\Delta_a}\tens\opb{q_{23}}\cor_{\Delta^\circ_a}\br\simeq\reim{q_{13}}\cor_{S_a}
	\eneqn
	Let $(x_1,x_3) \in X_1 \times X_3$ and set $Z_a=\opb{q_{13}}(x_1,x_3)\cap S_a$. Then $Z_a=B_a(x_1)\cap B_a^\circ(x_3)$ and it follows from the hypothesis that $(\reim{q_{13}}\cor_{S_a})_{(x_1,x_3)}\simeq\rsect(X_2;\cor_{Z_a})\simeq0$  for  $x_1\neq x_3$.  
	Therefore, $\reim{q_{13}}\cor_{S_a}$ is supported by $\Delta\subset X_{13}$ and we get 
	\eqn
	\reim{q_{13}}(\cor_{S_a}\tens\opb{q_2}\omega_X)
	&\simeq&\reim{q_{13}}((\cor_{S_a}\tens\opb{q_{13}}\cor_\Delta)\tens\opb{q_2}\omega_X)\\
	&\simeq&\reim{q_{13}}(\cor_{S_a\cap \opb{q_{13}}\Delta}\tens\opb{q_2}\omega_X)
	\isoto\cor_\Delta.
	\eneqn
	The last isomorphism is associated with the morphism  
	$\cor_{Z_a\cap \opb{q_{13}}\Delta}\tens\opb{q_2}\omega_X\to\epb{q_{13}}\cor_\Delta$ which is deduced from the morphism  $\cor_{Z_a\cap \opb{q_{13}}\Delta}\to \opb{q_{13}}\cor_\Delta$. (Recall that $Z_a\cap\opb{q_{13}}\Delta$ is open in $\opb{q_{13}}\Delta$.)
\end{proof}

For $0\leq a\leq\alpha_X$ set $\stD_a=\cor_{\Delta_a}$ and for 
$0<a\leq\alpha_X$, set $\stD_{-a}=\cor_{\Delta^\circ_a}\tens\opb{q_2}\omega_X$. 

\begin{lemma}\label{lem:add3b}
	Let $(X,d_X,\alpha_X)$ be a good metric space  satisfying~\eqref{hyp:dist2}.
	The map $a\mapsto\stD_a$ defines a
	monoidal presheaf on $[-\alpha_X,\alpha_X]$ with values in the monoidal category 
	$(\Derb(\cor_{X\times X}),\conv)$. 
\end{lemma}
\begin{proof}
	(i) For $0<b\leq a$,
	$\cor_{\Delta_a}\conv(\cor_{\Delta^\circ_b}\tens\opb{q_2}\omega_X)\simeq\cor_{\Delta_{a-b}}$.
	This follows from Lemmas~\ref{lem:add1} and~\ref{lem:add3} and
	$\cor_{\Delta_a}\conv\cor_{\Delta^\circ_b}\simeq\cor_{\Delta_{a-b}}\conv\cor_{\Delta_b}\conv \cor_{\Delta^\circ_b}$.
	
	\spa
	(ii) For $0<a,b,a+b<\alpha_X$,
	$\cor_{\Delta^\circ_b}\conv(\cor_{\Delta^\circ_a}\tens\opb{q_2}\omega_X)\simeq\cor_{\Delta^\circ_{a+b}}$.
	This follows from (i), Lemma~\ref{lem:add1} and 
	$(\cor_{\Delta^\circ_b}\tens\opb{q_2}\omega_X)\conv(\cor_{\Delta^\circ_a}\tens\opb{q_2}\omega_X)\ltens\cor_{\Delta_{a+b}}\simeq
	\cor_{\Delta}$.
	
	\spa
	(iii) For $0<b\leq a\leq\alpha_X$,
	$\cor_{\Delta^\circ_a}\conv\cor_{\Delta_b}\simeq\cor_{\Delta^\circ_{a-b}}$.
	Indeed, apply $\cor_{\Delta_{a-b}^\circ}\ltens\opb{q_{2}}\omega_X\conv$ to both sides of (ii).
\end{proof}

Applying  Theorem~\ref{th:pshK}, we get:
\begin{proposition}\label{pro:bithick1}
	Let $(X,d_X,\alpha_X)$ be a good metric space satisfying~\eqref{hyp:dist2}. 
	Then $\stD$ extends as a metric bi-thickening kernel and, for $0<a\leq\alpha_X$, one has 
	$\stD_{-a}\simeq\cor_{\Delta_{a}^\circ}\tens\opb{q_2}\omega_X$. Moreover, $\stR_a\simeq\stD_{-a}\conv$ for $a\geq0$.
\end{proposition}
There is indeed a better result. Set
\eq\label{not:Ipm}
&&
I=(-\alpha_X,\alpha_X). 
\eneq
\begin{theorem}\label{th:bithick3}
	Let $(X,d_X,\alpha_X)$ be a good metric space satisfying~\eqref{hyp:dist2}.
	There exists an object  $K^d\in\Derb(\cor_{X\times X\times I})$ and a distinguished triangle
	\eq\label{eq:GKSdt}
	&&\cor_{\{d_X(x,y)<-t\}}\tens\opb{q_2}\omega_X\to K^d\to  \cor_{\{d_X(x,y)\leq t\}}\xrightarrow[{\ \psi\ }]{+1}.
	\eneq
	In particular, 
	$K^d\vert_{\{t=a\}}\simeq\stK_{a}$  for $a\in I$.
\end{theorem}
\begin{proof}
	We shall  mimick the construction in \cite{GKS12}*{Exa.~3.10}.
	We have the isomorphism
	\eq\label{eq:GKS1}
	&&\rhom(\cor_{\Delta\times\{t=0\}},\cor_{X\times X\times\R})
	\simeq\cor_{\Delta\times\{t=0\}}\tens\opb{q_2}\omega^{\otimes-1}_{X}\,[-1].
	\eneq
	Indeed, $\cor_{\Delta\times\{t=0\}}\simeq \cor_{\Delta}\etens \cor_{\{t=0\}}$ and it follows from~\cite{KS90}*{Prop.~3.4.4} that 
	$\RD'_{X\times X\times\R}( \cor_{\Delta}\etens \cor_{\{t=0\}})\simeq \RD'_{X\times X} \cor_{\Delta}\etens \RD'_\R\cor_{\{t=0\}}$. Moreover, 
	$\RD'_{X\times X} \cor_{\Delta}\simeq \eim{\delta_X} \epb{\delta_X}\cor_{X\times X}\simeq
	\cor_{\Delta}\tens \opb{q_2}\omega_X$ and $ \RD'_\R\cor_{\{t=0\}}\simeq \cor_{\{t=0\}}\,[-1]$. 
	
	By Lemma~\ref{lem:add2}, we also have the isomorphism 
	\eq
	&&\rhom(\cor_{\{d_X(x,y)\leq-t\}},\cor_{X\times X\times\R})\simeq \cor_{\{d_X(x,y)<-t\}}\quad t\in (-a,0)\label{eq:GKS2}. 
	\eneq
	
	These isomorphisms 
	together with the morphism $\cor_{\{d_X(x,y)\leq-t\}}\to\cor_{\Delta\times\{t=0\}}$
	induce the morphism
	$\cor_{\Delta\times\{t=0\}}\tens\opb{q_2}\omega^{\otimes-1}_X[-1]\to \cor_{\{d_X(x,y)<-t\}}$.
	Hence, we obtain
	\eqn
	&&\cor_{\{d_X(x,y)\leq t\}}\to \cor_{\Delta\times\{t=0\}}\to\cor_{\{d_X(x,y)<-t\}}\tens\opb{q_2}\omega_X[+1]
	\eneqn
	Denoting by  $\psi$ the composition, we get the  distinguished triangle~\eqref{eq:GKSdt}.
\end{proof}

\begin{remark}\label{rem:tam}
	It would be possible to extend $K$ to a sheaf $\stKd\in\Derlb(\cor_{X\times X\times \R})$  by using Theorem~\ref{th:pshK}
	and using the monoidal category $(\Derb(\cor_{X\times X\times\R}),\sconv)$, where $\sconv$ is an operation adapted from~\cite{Ta08}, composition with respect  to $X$ and convolution with respect to $\R$.
\end{remark}

\subsection{Properties of the interleaving distance}

We shall extend to  metric spaces  a few results of~\cite{KS18}*{\S~2.2}. In this section, $(X,d_X)$ is a good metric space and $\stK$ is the metric thickening of the diagonal. Recall the interleaving distance $\dist_\stK$ of Definition~\ref{def:convdist}.
We set
\eq\label{def:distX}
&&\dist_X=\dist_\stK.
\eneq

\begin{lemma}\label{le:sectRaF}
	Let $F\in\Derb(\cor_X)$ and let $a\geq0$. Then
	\begin{align*}	
		\rsect(X;\stk_a\conv F)\isoto\rsect(X;F) \quad \textnormal{and} \quad \rsect_c(X;\stk_a\conv F)\isoto\rsect_c(X;F).
	\end{align*}
\end{lemma}
\begin{proof}
	It follows from the definition of the functor $\stk_a$ that is it enough to check these isomorphisms for $0\leq a\leq\alpha_X$,  thus replacing $\stk_a$ with $\cor_{\Delta_a}$. Consider the Cartesian diagram
	\eqn
	&&\xymatrix@R=2ex@C=2ex{
		&X\times X\ar[ld]_-{q_1}\ar[rd]^-{q_2}&\\
		X\ar[rd]_-{q'_2}&&X\ar[ld]^-{q'_1}\\
		&\rmpt&
	}\eneqn
	Using the fact that $q_1$ and $q_2$ are proper on $\Delta_a$ we get the isomorphisms 
	\eqn
	\rsect(X;\cor_{\Delta_a}\conv F)&\simeq&\roim{q'_2}\reim{q_1}(\cor_{\Delta_a}\ltens\opb{q_2}F)\simeq
	\roim{q'_2}\roim{q_1}(\cor_{\Delta_a}\ltens\opb{q_2}F)\\
	&\simeq&\roim{q'_1}\roim{q_2}(\cor_{\Delta_a}\ltens\opb{q_2}F)\simeq\roim{q'_1}\reim{q_2}(\cor_{\Delta_a}\ltens\opb{q_2}F)\\
	&\simeq&\roim{q'_1}(\reim{q_2}\cor_{\Delta_a}\ltens F)\\
	&\simeq&\roim{q'_1}F\simeq\rsect(X;F).
	\eneqn
	Here we use the isomorphism $\reim{q_2}\cor_{\Delta_a}\simeq\cor_X$ which follows from the fact that the fibers of $q_2\cl\Delta_a\to X$ are compact and contractible. 
	
	A similar  proof holds for $\rsect_c(X;F)$. 
\end{proof}

\begin{proposition}\label{pro:sectF0}
	Let $F,G\in\Derb(\cor_X)$. If $\dist_X(F,G)<+\infty$,  then $\rsect(X;F)\simeq\rsect(X;G)$ and $\rsect_c(X;F)\simeq\rsect_c(X;G)$.	
\end{proposition}
\begin{proof}
	This follows immediately from the definition of the distance and Lemma~\ref{le:sectRaF}.
\end{proof}

\begin{proposition}\label{pro:sectF}
	Let $F\in\Derb(\cor_X)$ and assume that $\supp(F)\subset B(x_0,a)$ with $a \leq\alpha_X$. Set $M=\rsect(X;F)$ and denote by $M_{x_0}$ the sky-scraper sheaf at $\{x_0\}$ with stalk $M$. Then $\dist_X(F,M_{x_0})\leq a$.
\end{proposition}
We shall mimick the proof of~\cite{KS18}*{Exa.~2.4}.
\begin{proof}
	We have
	\eqn
	&&\cor_{\Delta_a}\conv M_{x_0}\simeq M_{B(x_0,a)},
	\eneqn
	the constant sheaf on $B(x_0,a)$ with stalk $M$ extended by $0$ outside of $B(x_0,a)$.
	
	Denote by $a_X\cl X\to\rmpt$ the unique map from $X$ to $\rmpt$. The morphism $\opb{a_X}\roim{a_X}F\to F$ defines the map
	$M_X\to F$ and $F$ being supported in $B(x_0,a)$, we get the morphism 
	$g\cl \cor_{\Delta_a}\conv M_{x_0}\simeq  M_{B(x_0,a)}\to F$. 
	
	On the other hand, we have 
	\eq\label{eq:sectL}
	&&\ba{l}
	(\cor_{\Delta_a}\conv F)_{x_0}\simeq\rsect(\opb{q_1}(x_0);\cor_{\Delta_a}\ltens\opb{q_2}F)\\
	\hspace{2.1cm}\simeq \rsect(\{x_0\}\times X;\{x_0\}\times\cor_{B(x_0,a)}\ltens \opb{q_2}F)\\
	\hspace{2.1cm}\simeq\rsect(B(x_0,a);F)\simeq M
	\ea\eneq
	which defines  $f\cl \cor_{\Delta_a}\conv F\to M_{x_0}$.
	One easily checks that $f$ and $g$ satisfy the compatibility conditions in Definition~\ref{def:convdist}. Therefore
	$\dist_X(F,M_{x_0})\leq a$.
\end{proof}
In particular, a non-zero object can be $a$-isomorphic (see Definition~\ref{def:convdist}) to
the zero object. 

\begin{corollary}\label{cor:distsuppcpt}
	Let $F,G\in\Derb(\cor_X)$ and assume that there exists a ball $B_{x_0}(a)$ with $a \leq\alpha_X$ which contains the supports of $F$ and $G$. 
	Then $\dist_X(F,G)<\infty$ if and only if $\rsect(X;F)\simeq\rsect(X;G)$.
\end{corollary}
\begin{proof}
	(i) Assume $M \eqdot \rsect(X;F)\simeq\rsect(X;G)$. Then  
	\begin{equation*}
		\dist_X(F,G)\leq \dist_X(F,M_{x_0})+\dist_X(G,M_{x_0})
	\end{equation*}
	and it remains to apply Proposition~\ref{pro:sectF}.
	
	\spa
	(ii) The converse assertion is nothing but Proposition~\ref{pro:sectF0}.
\end{proof}

\begin{corollary}\label{cor:distdt}
	Consider two distinguished triangles $F_1\to F_2\to F_3\to[+1]$ and $G_1\to G_2\to G_3\to[+1]$ in $\Derb(\cor_X)$. 
	Assume that  there exists a ball $B_{x_0}(a)$ with $a \leq\alpha_X$  which contains the supports of all sheaves $F_i,G_i$ \lp$i=1,2,3$\rp\, and also assume that $\dist_X(F_i,G_i)<\infty$ for $i=1,2$. Then
	$\dist_X(F_3,G_3)<\infty$.
\end{corollary}
\begin{proof}
	It follows from  Corollary~\ref{cor:distsuppcpt} that $\rsect(X;F_i)\simeq\rsect(X;G_i)$ for $i=1,2$. Since the functor 
	$\rsect(X;\scbul)$ is triangulated, this isomorphism still holds for $i=3$. Then the result follows again from 
	Corollary~\ref{cor:distsuppcpt}.
\end{proof}

\subsubsection*{Locally constant sheaves}

Recall that an object $L\in\Derb(\cor_X)$ is locally constant (resp.\ constant)  if, for all $j\in\Z$, $H^j(L)$ is a locally constant (resp.\ constant) sheaf.

\begin{lemma}\label{lem:locsyst}
	Let $L \in\Derb(\cor_X)$ and assume that $L$ is locally constant.  Let $a\geq0$. Then $\stk_a\conv L\isoto L$. 
\end{lemma}	
\begin{proof}
	We may choose $a$ such that $a<\alpha_X$ and replace $\stk_a$ with $\cor_{\Delta_a}$. 
	It is then enough to prove that, for $x\in X$, the natural morphism $(\cor_{\Delta_a}\conv L )_{x}\to L_x $ is an isomorphism.
	We may also assume that $ L $ is a constant sheaf in a neighborhood of $B_a(x)$. Then by~\eqref{eq:sectL}, we get
	\eqn
	&&  (\cor_{\Delta_a}\conv  L )_{x}\simeq \rsect(B_a(x); L )\simeq  L_x.
	\eneqn
\end{proof}

\begin{proposition}\label{pro:locsyst}
	Let $F, G\in\Derb(\cor_X)$. Assume that  $F$ is locally constant and  that  $\dist_X(F,G)$ is finite.  Then $F$ is a direct summand of $G$.
	In particular, if both $F$ and $G$ are locally constant,  then  $F\simeq G$.
\end{proposition}
\begin{proof}
	By the hypothesis and Lemma~\ref{lem:locsyst} we have morphisms $F\to G\to F$ such that the composition is an isomorphism.
\end{proof}
It follows that the interleaving distance is not really interesting when considering locally constant sheaves.

\subsection{The stability theorem}
Let $X$ be a good topological space and let $(Y,d_Y)$ be a good metric space. 
We denote by $\stK^Y_a$ the kernel  on $Y\times Y$. It  defines an  endofunctor of
$\Derb(\cor_{X\times Y})$, $K\mapsto K\conv \stK_a$. 
We then get 
a pseudo-distance on  $\Derb(\cor_{X\times Y})$ that we call a relative distance and denote  by $\dist_{X\times Y/X}$.

\begin{theorem}[{The stability theorem}]\label{th:stabmetric}
	Let $X$ be a good topological space and let $(Y,d_Y)$ be a good metric space. 
	Let $K_1,K_2\in\Derb(\cor_{Y\times X})$ and let $F\in\Derb(\cor_X)$. Then
	\banum
	\item
	$\dist_{Y}( K_1\conv F,K_2\conv F)\leq \dist_{Y\times X/X}(K_1,K_2)$. 
	\item
	Assume moreover that $X$ and $Y$ are  $C^\infty$-manifolds  and  that $(Y,d_Y)$  satisfies~\eqref{hyp:dist2} and~\eqref{hyp:dist4}.  
	Then 
	$\dist_{Y}( K_1\npconv F,K_2\npconv F)\leq \dist_{Y\times X/X}(K_1,K_2)$. 
	\eanum
\end{theorem}
\begin{proof}
	(a) We have 
	\eqn
	&&\stD^Y_a\conv (K_i\conv F)\simeq (\stD^Y_a\conv K_i)\conv F,\quad i=1,2.
	\eneqn
	Then the result follows immediately from Definition~\ref{def:convdist}. 
	
	\spa
	(b) The proof is the same as in (a)  after replacing $\conv$ with $\npconv$ and using Proposition~\ref{pro:add4}.
\end{proof}

Let $X$ and $Y$ be as above and let $f_1,f_2\cl X\to Y$ be two continuous maps. As usual, one sets 
\eqn
&&\dist(f_1,f_2)=\sup_{x\in X}d_Y(f_1(x),f_2(x)).
\eneqn

\begin{corollary}[{(The metric stability theorem, see~\cite{KS18}*{Th.~2.7})}]\label{cor:stabmetric}
	Let $X$ be a good topological space and let $Y$ be a \lp real, finite dimensional\rp\, normed  vector space, $d_Y$ the associated distance.
	Then $\dist_{Y}(\reim{f_1}F,\reim{f_2}F)\leq \dist(f_1,f_2)$. If $X$ is a $C^\infty$-manifold and $Y$ is an Euclidian vector space, the same result holds with $\reim{f}$ replaced with $\roim{f}$. 
\end{corollary}
\begin{proof}
	Let $a=\dist(f_1,f_2)$. Of course, we may assume that $a<\infty$. 
	Denote by $\Gamma_i$  the graph of $f_i$ in $Y\times X$. Then
	\eq\label{eq:deltaconvgamma1}
	&&\Gamma_{f_i}\subset \Delta_a^Y\conv\Gamma_{f_j}, \, i,j\in\{1,2\}.
	\eneq
	Moreover, for $f=f_1$ or $f=f_2$, one has 
	\eq\label{eq:deltaconvgamma2}
	&&\cor_{\Delta_a^Y}\conv \cor_{\Gamma_f}\simeq \cor_{ \Delta_a^Y\conv\Gamma_f}.
	\eneq
	Set $K_i=\cor_{\Gamma_{f_i}}$ ($i=1,2$). 
	By~\eqref{eq:deltaconvgamma1} and~\eqref{eq:deltaconvgamma2} , we get 
	morphisms $\cor_{\Delta_a^Y}\conv K_{f_1}\to K_{f_2}$ and $\cor_{\Delta_a^Y}\conv K_{f_2}\to K_{f_1}$ satisfying the conditions of Definition~\ref{def:convdist}. 
	Therefore, 
	\eq\label{eq:distfdistK}
	&&\dist_{Y\times X/X}(K_{f_1},K_{f_2})\leq a= \dist(f_1,f_2).
	\eneq
	Since   $\reim{f_i}F\simeq K_i\conv F$ and $\roim{f_i}F\simeq K_i\npconv F$, 
	the result follows from  Theorem~\ref{th:stabmetric} since hypotheses~\eqref{hyp:dist2} and~\eqref{hyp:dist4} are satisfied
	if $Y$ is an Euclidian vector space. 
\end{proof}

\begin{remark}
	In~\cite{KS18}*{Th.~2.7} the proof  for $\roim{f}$ and  $\reim{f}$ is almost the same and  $X$ is only assumed to be a good topological space.  The reason why the non proper case is easier in the situation of~\cite{KS18} is that these authors use the convolution functor $\cor_{B_a}\star$ instead of $\cor_{\Delta_a}\conv$.
	
	More precisely, consider the diagram in which $Y$ is a real finite dimensional normed vector space,  $Y_1$ and $Y_2$ are two copies of $Y$ and    $s$ is the map $(y_1,y_2)\mapsto y_1+y_2$, $s_{13}$ is the map $(y_1,x,y_2)\mapsto (y_1+y_2,x)$:
	\eqn
	&&\xymatrix{
		&Y_{1}\times X\times Y_2\ar[ld]_-{p_{12}}\ar[d]^-{s_{13}}\ar[rd]^-{p_{23}}&\\
		Y_{12}\ar[rd]_-s&Y\times X\ar[d]_{p_1}\ar[rd]_-{p_2}&X\times Y_2\ar[d]_{q_1}\ar[rd]^-{q_2}\\
		&Y&X&Y_2.
	}\eneqn
	Let $F\in\Derb(\cor_X)$,  $K\in\Derb(\cor_{Y_2\times X})$ and denote by $B_a$ the closed ball of $Y_1$ with center $0$ and radius $a\geq0$. Set for short $\cor_B\eqdot\cor_{B_a}$. Then 
	\eqn
	\cor_B\star(K\npconv F)&\simeq& \roim{s}(\cor_B\etens\roim{q_2}(K\ltens\opb{q_1}F))\\
	&\simeq&\roim{s}\roim{p_{12}}(\cor_B\etens(K\ltens\opb{q_1}F))\\
	&\simeq&\roim{p_1}\roim{s_{13}}(\cor_B\etens(K\ltens\opb{q_1}F))\\
	&\simeq&\roim{p_1}\roim{s_{13}}((\cor_B\etens K)\ltens\opb{s_{13}}\opb{p_2}F)\\
	&\simeq&\roim{p_1}(\roim{s_{13}}(\cor_B\etens K)\ltens\opb{p_2}F)\simeq(\cor_B\star K)\npconv F.
	\eneqn
	Here, the 2nd isomorphism follows from the fact that $\cor_B$ being  cohomologically constructible,  the functor $\cor_B\etens\scbul$ commutes with (non proper) direct images thanks to~\cite{KS90}*{Prop.~3.4.4}.
	The 5th isomorphism follows from the fact that $s$ is proper on $\supp(\cor_B\etens K)$. 
\end{remark}

\subsection{Lipschitz kernels}

\subsubsection*{A general setting}
We consider two good metric spaces $(X,d_X)$ and $(Y,d_Y)$. 
To avoid confusion, we denote by  $\alpha_X$ and $\alpha_Y$ the constants appearing in~\eqref{hyp:dist1},  by $\Delta^X_a$ and $\Delta^Y_a$ the thickenings of the diagonals, by $\stDX_a$ and  $\stDY_a$ the associated thickening kernels and by $\rho^X_{a,b}$ and 
$\rho^Y_{a,b}$ the restriction functors. 
Recall the  notation for $F\in\Derb(\cor_X)$
\eqn
&&\Phi_K(F)=K\conv F.
\eneqn

\begin{definition}\label{def:lipFct}
	Let $\delta>0$ and let $K\in\Derb(\cor_{Y\times X})$. 
	We say that $K$ is a $\delta$-Lipschitz kernel  from $X$ to $Y$ if  there exists $\rho>0$ such that $\rho\leq\alpha_X$ and $\delta\rho\leq\alpha_Y$ and
	there are  morphisms of sheaves 
	$\sigma_a\cl \stDY_{\delta a}\conv K\to K\conv \stDX_a$ for $0\leq a\leq\rho$  satisfying the following compatibility relations:
	\begin{nnum} 
		\item for $0\leq a\leq b \leq\rho$,  the diagram of sheaves  commutes:
		\eq\label{eq:sqdiagcom0}
		&&\ba{l}\xymatrix{
			\stDY_{\delta b}\conv K \ar[d]_-{\rho^Y_{\delta a,\delta b}}\ar[r]^-{\sigma_b}&K\conv \stDX_b\ar[d]^-{\rho^X_{a,b}}\\\
			\stDY_{\delta a}\conv K\ar[r]^-{\sigma_a}& K\conv\stDX_a,
		}
		\ea\eneq
		\item  for  $0\leq a, \; b$ and $a+b\leq\rho$, the diagram of sheaves commutes: 
		\eq\label{eq:sqdiagcom1}
		&&\ba{l}\xymatrix{
			\stDY_{\delta (a+b)}\conv K\ar[rr]^-{\stDY_{\delta b}\conv \sigma_a} \ar@/_1.5pc/[rrrr]_-{\sigma_{a+b}}&& \stDY_{\delta b}\conv K\conv \stDX_{a} \ar[rr]^-{\sigma_b\conv\stDX_a} && K\conv \stDX_{a+b}.
		}
		\ea\eneq
	\end{nnum}
	
	A Lipschitz kernel  is a  $\delta$-Lipschitz kernel for some $\delta >0$.
\end{definition}
Note that thanks to the hypothesis that $a\leq\alpha_X$, we could have written $\cor_{\Delta^X_a}$ instead of $\stDX_a$ and
similarly with $Y$ instead of $X$. We have chosen to use the notation $\stD$ thanks to the next lemma.
\begin{remark}\label{rem:danger}
	Of course,  a  Lipschitz kernel form $X$ to $Y$ is not necessarily a Lipschitz kernel from $Y$ to $X$. However, when there is no risk of confusion, we shall simply call $K$ ``a Lipschitz kernel''.
\end{remark}
\begin{lemma}
	If $K$ is a Lipschitz kernel, then  for all $a\geq0$ there are  morphisms of sheaves 
	$\sigma_a\cl \stDY_{\delta a}\conv K\to K\conv \stDX_a$  and moreover~\eqref{eq:sqdiagcom0} and~\eqref{eq:sqdiagcom1} are satisfied for all $a,b\geq0$.
\end{lemma}
\begin{proof}[Sketch of proof]
	Assume we  have constructed the morphisms $\sigma_a$ for $a\leq A$ and let $0\leq b\leq\rho$. One defines the morphism
	\eqn
	\sigma_{a+b}\cl  \stDY_{\delta (a+b)}\conv K&\simeq& \cor_{\Delta_{\delta b}^Y}\conv\stDY_{\delta (a)}\conv K\\
	&\to& \cor_{\Delta_{\delta b}^Y}\conv K \conv\stDX_a\\
	&\to& K\conv  \cor_{\Delta_{b}^X}\conv\stDX_a\simeq K\conv \stDX_{a+b}.
	\eneqn
	The fact that $\sigma_a$ is well-defined and the verification of the compatibility relations~\eqref{eq:sqdiagcom0} and~\eqref{eq:sqdiagcom1} are left to the reader.
\end{proof}

The next result is essentially a reformulation in the language of kernels of~\cite{SMS18}*{Th.~4.3}.

\begin{theorem}[{(The functorial Lipschitz theorem})]\label{th:FctLip}
	Let $(X,d_X)$ and $(Y,d_Y)$ be good metric spaces and let 
	$K\in\Derb(\cor_{Y\times X})$ be a $\delta $-Lipschitz kernel from $X$ to $Y$. Let  $F_1,F_2\in\Derb(\cor_X)$.
	\banum
	\item
	One has 
	$\dist_{Y}(K\conv F_1, K\conv F_2)\leq \delta \cdot \dist_{X}(F_1,F_2)$.
	\item
	Assume moreover that $X$ and $Y$ are  $C^\infty$-manifolds  satisfying~\eqref{hyp:dist2} and~\eqref{hyp:dist4}. \\
	Then $\dist_{Y}(K\npconv F_1, K\npconv F_2)\leq \delta \cdot \dist_{X}(F_1,F_2)$.
	\eanum
\end{theorem}

\begin{proof}
	(a) Let $F_1,F_2\in\Derb(\cor_X)$  and assume that $F_1$ and $F_2$ are $a$-isomorphic.  
	Hence, there are morphisms 
	\eqn
	&&f\cl \stDX_{a}\conv F_1\to F_2,\quad g\cl \stDX_{a}\conv F_2\to F_1
	\eneqn
	satisfying the conditions of Definition~\ref{def:convdist}. Applying the functor $K\conv$ we get the morphisms given by the dotted arrows
	\eqn
	&&\xymatrix{
		K\conv\stDX_{a}\conv F_1\ar[r]^-{\Phi_K(f)}&K\conv F_2\\
		\stDY_{\delta a}\conv K\conv F_1\ar[u]^{\sigma_a}\ar@{.>}[ru]&
	}\quad
	\xymatrix{
		K\conv\stDX_{a}\conv F_2\ar[r]^-{\Phi_K(g)}&K\conv F_1\\
		\stDY_{\delta a}\circ K\conv F_2\ar[u]^{\sigma_a}\ar@{.>}[ru]&
	}
	\eneqn
	
	Now consider the diagram 
	\eqn
	&&\xymatrix{
		K\conv\stD_{2a}\conv F_1\ar[rr]^{\Phi_K(\stL_a(f))}&&K\conv \stDX_{a}\conv F_2\ar[r]^-{\Phi_K(g)}& K\conv F_1.\\
		\stDY_{\delta a}\conv K\conv \stDX_{a}\conv F_1\ar[rr]^-{\stL^Y_{\delta a}(\Phi_K(f))}\ar[u]^-{ \stL^X_a(\sigma_a)}&&\stDY_{\delta a}\conv K\conv F_2\ar[u]^-{\sigma_a}\ar@{.>}[ru]&\\
		\stDY_{2\delta a}\conv K\conv F_1\ar[u]^{ \stL^Y_{\delta a}(\sigma_a)}\ar@{.>}[rru]&&&
	}
	\eneqn
	The two  diagrams with dotted arrows  commute by the definition of the dotted arrows and the square diagram  commutes by Definition~\ref{def:lipFct} (i).
	The composition of the two vertical arrows is given by $\sigma_{2a}$ by Definition~\ref{def:lipFct} (ii). The composition of the two horizontal arrows is given by $\rho^X_{0,2a}$. 
	Therefore, the composition of the two dotted arrows is given by $\rho^X_{0,2a}\sigma_{2a}=\rho^Y_{0,2\delta a}$.
	The same result holds when interchanging the roles of $F_1$ and $F_2$.
	
	\spa
	(b) The proof is the same as in (a)  after replacing $\conv$ with $\npconv$ 
	and using Proposition~\ref{pro:add4}.
\end{proof}

In particular, we get:

\begin{corollary}\label{cor:FctLip}
	Assume that $K\in\Derb(\cor_{Y\times X})$ is a $\delta$-Lipschitz kernel from $X$ to $Y$ and that there exists a $\opb{\delta}$-Lipschitz kernel 
	$L\in\Derb(\cor_{X\times Y})$ from $Y$ to $X$ such that $\Phi_{L\conv K}\simeq\id_{\Derb(\cor_X)}$. 
	Then for $F_1,F_2\in\Derb(\cor_X)$, one has 
	$\dist_{Y}(K\conv F_1, K\conv F_2)= \delta \cdot\dist_{X}(F_1,F_2)$.
	
	If $X$ and $Y$ are  $C^\infty$-manifolds  satisfying~\eqref{hyp:dist2} and~\eqref{hyp:dist4}, then the same result holds for 
	$K\conv F$ replaced with $K\npconv F$. 
\end{corollary}

\subsubsection*{Lipschitz correspondences}
As above, we denote by $X_i$ and $Y_i$ ($i=1,2$) two copies of $X$ or $Y$. 
We keep the assumptions and notations of the beginning of this section.

We assume to be given a subset $S$ of $Y\times X$ and   consider the diagram

\eq\label{diag:Lip}
&&\xymatrix@R=3ex@C=2ex{
	&Y_{12}\times X_1\ar[ld]_-{p_{12}}\ar[rd]^-{p_{23}}\ar[dd]|-{p_{13}}&&Y_{2}\times X_{12}\ar[ld]_-{q_{12}}\ar[rd]^-{q_{23}}\ar[dd]|-{q_{13}}&\\
	\Delta^Y_b\subset Y_{12}&&S\subset Y_2\times X_1&&\Delta_a^X\subset X_{12}\\
	&Y_1\times X_1\ar@{=}[rr]&&Y_2\times X_2&
}\eneq

We set
\eqn
\Delta_b^Y\times_YS=\opb{p_{12}}(\Delta_b^Y)\cap\opb{p_{23}}(S)\subset Y_{12}\times X_1, \; S\times_X\Delta_a^X=\opb{q_{12}}(S)\cap\opb{q_{23}}(\Delta^X_a)\subset Y_2\times X_{12}.
\eneqn
Note that $\Delta_b^Y\conv S=p_{13}(\Delta_b^Y\times_YS)$ and $S\conv\Delta_a^X=q_{13}(S\times_X\Delta_a^X)$ 
are contained in $Y_1\times X_1=Y_2\times X_2=Y\times X$.
We shall consider one of  the hypotheses~\eqref{hyp:dist22} or~\eqref{hyp:dist33} below
for some  constants $\rho, \delta>0$ such that $ \rho \leq \alpha_X$ and $\delta \rho \leq \alpha_Y$.
\eq\label{hyp:dist22}
&&\left\{\parbox{75ex}{
	(a) $S$ is a closed subset of $Y\times X$,\\
	(b)  the  fibers of the projection $p_{13}\cl\Delta_b^Y\times_YS\to\Delta_b^Y\conv S$  are contractible or empty for $0\leq b\leq\alpha_Y$,\\
	(c) $ S\conv\Delta^X_a\subset \Delta_{\delta a}^Y\conv S$  for $a\leq\rho $.
}\right.
\eneq
\eq\label{hyp:dist33}
&&\left\{\parbox{75ex}{
	(a) $S$ is a closed subset of $Y\times X$,\\
	(b) there a  closed embedding $\iota\cl Y_2\times X_{12}\into Y_{12}\times X_1$ such that 
	$p_{13}\circ\iota=q_{13}$,\\
	(c) $ \iota(S\times_X\Delta^X_a)\subset \Delta_{\delta a}^Y\times_Y S$ for $a\leq\rho $.
}\right.
\eneq

\begin{theorem}\label{th:metricLip}
	Let $S\subset Y\times X$ and consider constants $\rho, \delta>0$ such that $ \rho \leq \alpha_X$ and $\delta \rho \leq \alpha_Y$. 
	One makes either hypothesis~\eqref{hyp:dist22} or  hypothesis~\eqref{hyp:dist33}.
	Then $\cor_S\in \Derb(\cor_{Y\times X})$ is a $\delta$-Lipschitz kernel from $X$ to $Y$.
\end{theorem}
\begin{proof}
	(i)  It is enough to construct  a natural morphism  of sheaves
	\eq\label{eq:morYSSX}
	&&\cor_{\Delta^Y_{\delta a}}\conv\cor_S\to\cor_S\conv \cor_{\Delta^X_a}\mbox{ for $a\leq \rho$ 
		(which implies $\delta a\leq\alpha_Y$}).
	\eneq
	
	\spa
	(ii)--(a) Assume~\eqref{hyp:dist22}. Since the closed set  $\Delta_{\delta a}^Y\conv S$ contains the closed set $S\conv\Delta^X_a$, we have a morphism of sheaves
	\eq\label{eq:morYSX}
	&&\cor_{\Delta_{\delta a}^Y\conv S}\to\cor_{S\conv\Delta^X_a}.
	\eneq
	By~Lemma~\ref{le:kAcB} and the hypothesis, there is an isomorphisms and a morphism
	\eqn
	&&\cor_{\Delta_{\delta a}^Y\conv S}\simeq \cor_{\Delta_{\delta a}^Y}\conv \cor_S,\quad \cor_{S\conv\Delta^X_a}\to \cor_{S}\conv\cor_{\Delta^X_a}.
	\eneqn
	Together with~\eqref{eq:morYSX}, this defines~\eqref{eq:morYSSX}.
	
	\spa
	(ii) --(b)  Assume~\eqref{hyp:dist33}. By this hypothesis, there is a natural morphism 
	\eq\label{eq:RXconvSB}
	&& \cor_{\Delta_{\delta a}^Y\times_Y S}\to \oim{\iota}\cor_{S\times_X\Delta_a^X}.
	\eneq
	Now remark that 
	\eqn
	&&\cor_{S\times_X\Delta_a^X}\simeq \opb{q_{12}}\cor_S\ltens\opb{q_{23}}\cor_{\Delta_a^X},\quad \cor_{\Delta_{\delta a}^Y\times_Y S}\simeq \opb{p_{12}}\cor_{\Delta_a^X}\ltens\opb{p_{23}} \cor_S.
	\eneqn
	By~\eqref{eq:RXconvSB}, we get the morphisms
	\eqn
	\cor_{\Delta_{\delta a}^Y}\conv\cor_S&\simeq& \reim{p_{13}} (\opb{p_{12}}\cor_{\Delta_a^X}\ltens\opb{p_{23}} \cor_S)
	\simeq \reim{p_{13}} \cor_{\Delta_{\delta a}^Y\times_Y S}\\
	&\to&\reim{p_{13}} \oim{\iota} \cor_{S\times_X\Delta_a^X}
	\simeq\reim{p_{13}} \oim{\iota}(\opb{q_{12}}\cor_S\ltens\opb{q_{23}}\cor_{\Delta_a^X})\\
	&\simeq&\reim{q_{13}}(\opb{q_{12}}\cor_S\ltens\opb{q_{23}}\cor_{\Delta_a^X})
	\simeq \cor_S\conv\cor_{\Delta_a^X}.
	\eneqn
	We have thus constructed the morphism~\eqref{eq:morYSSX}. 
\end{proof}

Let $f \cl X\to Y$  be a continuous map. We set $\Gamma_f= \{(f(x),x) \in Y \times X \}$.

\begin{corollary}\label{cor:Lip}
	Let $f \colon (X,d_X) \to (Y, d_Y)$ be a $\delta$-Lipschitz map. Then $\cor_{\Gamma_f}$ is a $\delta$-Lipschitz kernel from $X$ to $Y$.
\end{corollary}
\begin{proof}
	(i) We shall check~\eqref{hyp:dist22} with $S=\Gamma_f$. Of course, this set  is closed in $Y\times X$.
	
	\spa
	(ii) Let us check~\eqref{hyp:dist22}~(b). One has 
	\eqn
	&&\Delta_b^Y\times_YS=\{(y_1,y_2,x)\in Y \times Y\times X;d_Y(y_1,y_2)\leq b, y_2=f(x)\}.
	\eneqn
	For $(y_1,x)\in \Delta_b^Y\conv S$, $\opb{q_{13}}(y_1,x)\cap \Delta_b^Y\times_YS$ is the set $y_2=\{f(x)\}$ if 
	$d_Y(y_1,y_2)\leq b$ and is empty otherwise.
	
	\spa
	(iii) Let us check~\eqref{hyp:dist22}~(c). One has
	\eqn
	&&\Delta_{\delta a}^Y\conv S=\{(y,x)\in Y\times X; \exists y'\in Y,\, d_Y(y,y')\leq\delta a, \,y'=f(x)\},\\
	&&S\conv\Delta^X_a.=\{(y,x)\in Y\times X; \exists x'\in X,\, d_X(x,x')\leq a, \, y=f(x')\}.
	\eneqn
	Let $(y,x)\in S\conv\Delta^X_a$ and let $x'\in X$ be such that $d_X(x,x')\leq a, \, y=f(x')$. Set $y'=f(x)$. Then $d_Y(y,y')\leq \delta a$ since $f$ is $\delta$-Lipschitz and therefore $(y,x)\in \Delta_{\delta a}^Y\conv S$.
\end{proof}

\begin{example}
	Let $X=\BBS^1$, $Y=\R^2$ and denote by $S$ the graph of the embedding 
	$j\cl\BBS^1\into\R^2$. Then  $\cor_S \in\Derb(\cor_{Y\times X})$ is a $\delta$-Lipschitz kernel from $X$ to $Y$ with
	$\delta=\frac{\pi}{ \sqrt{2}}$ and defines a fully faithful functor. 
\end{example}

\begin{corollary}\label{cor:lipfun}
	Let $(X,d_X)$ and $(Y,d_Y)$ be good metric spaces and let 
	$f \colon X \to Y$ be a $\delta $-Lipschitz map. Let  $F_1,F_2\in\Derb(\cor_X)$.
	\banum
	\item
	One has 
	$\dist_{Y}(\reim{f}F_1, \reim{f}F_2)\leq \delta \cdot \dist_{X}(F_1,F_2)$.
	\item
	If moreover, $X$ and $Y$ are  $C^\infty$-manifolds satisfying hypotheses~\eqref{hyp:dist2} and~\eqref{hyp:dist4},
	then \\
	$\dist_{Y}(\roim{f} F_1, \roim{f} F_2)\leq \delta \cdot \dist_{X}(F_1,F_2).$
	\eanum
\end{corollary}

\begin{proof}
	First remark that for every $F \in \Derb(\cor_{X})$, $\reim{f} F \simeq \cor_{\Gamma_f} \conv F$ and $\roim{f} F \simeq \cor_{\Gamma_f} \npconv F$. Then apply Corollary \ref{cor:Lip} and Theorem \ref{th:FctLip}.	
\end{proof}

\subsection{Some elementary examples}\label{sect:exa}

\subsubsection*{Vector spaces}
The interleaving distance for sheaves on a (finite dimensional) real normed vector space has been studied with great details in~\cite{KS18}
and in fact this paper is a special case and a guide for the present one. In loc.\ cit.\, the composition $\cor_{\Delta_a}\conv$ was replaced by the convolution $\cor_{B_a}\star$ which  is  equivalent (see Proposition ~\ref{pro:conv}). When the norm is not Euclidian, we get an example where the whole theory developed here applies although the metric space  is not associated with a Riemannian manifold. 

The next result is obvious.
\begin{proposition}\label{th:normedtspace}
	Let $X=\BBV$ be a real finite dimensional Euclidian vector space
	and let  $d_X$ be the associated distance. Then $(X,d_X)$ satisfies
	hypotheses~\eqref{hyp:dist1},~\eqref{hyp:dist2} and~\eqref{hyp:dist4}.
\end{proposition}

In the situation of Proposition~\ref{th:normedtspace},  the bi-thickening kernel
is given by
\eqn\label{eq:bitheuclidian}
\stD_a&\simeq&
\begin{cases}
	\cor_{\Delta_a }\text{ if } a\geq0,\\
	\cor_{\Delta^\circ_{-a}}\,[n] \text{ if }a<0.
\end{cases}
\eneqn
More precisely, in this situation, the sheaf $\stKd$ is described, up to isomorphism, in~\cite{GKS12}*{Exa.~3.11} by the distinguished triangle in $\Derb(\cor_{\R^n\times \R^n\times \R})$:
\eqn
&&\cor_{\{\vert x-y\vert<-t\}}[n]\to \stKd\to 
\cor_{\{\vert x-y\vert\leq t\}}\xrightarrow[{}]{+1}
\eneqn

\subsubsection*{The real line}
Let $X=\R$ be the real line. Recall that, $\cor$ being a field, one has an isomorphism

\eq\label{eq:shvonR}
&&F\simeq\bigoplus_jH^j(F)\,[-j]\mbox{ for } F\in\Derb(\cor_X).
\eneq
Hence, the study of objects of $\Derb(\cor_X)$ is reduced to that of objects of $\md[\cor_X]$. But, as it is well-known, there exist non zero morphisms 
between objects concentrated in different degrees. 

Constructible sheaves with compact support on $\R$ (over a field) are classified via the famous theorem of  Crawley-Boevey~\cite{CB14}.
See also~\cite{Gu19} for a formulation in the language of constructible sheaves and see~\cite{KS18}*{Th.~1.17} for the case   of not necessarily compactly supported sheaves.
Distances on such sheaves are studied with great details in~\cite{BG18}.
Recall that in this setting the thickening of the identity is provided by the following family of endofunctors of $\Derb(\cor_\R)$,
$\cor_{B_a} \star$, $a\geq0$, where $B_a=[-a,a]$.

\subsection{Example: the Fourier-Sato transform}\label{sect:FS}

Consider first the topological $n$-sphere  ($n>0$) defined as follows. Let $\BBV$ be a real vector space of dimension $n+1$, set $\dot\BBV=\BBV\setminus\{0\}$ and $\BS\eqdot\dot\BBV/\R^+$ where $\R^+$ is the multiplicative group $\R_{>0}$. Define similarly the dual sphere $\BS^*$, starting with $\BBV^*$. The sets
\eq\label{eq:PI}
&&P=\{(y,x)\in\BS^*\times\BS;\langle y,x\rangle\geq0\},\quad I=\{(y,x)\in\BS^*\times\BS;\langle y,x\rangle>0\},
\eneq
are well-defined. 
We define the kernel
\eq\label{eq:kerI}
&&K_I=\cor_I\ltens(\omega_{\BS^*}\etens\cor_\BS).
\eneq
Note that $K_I\simeq \rhom(\cor_P,\omega_{\BS^*}\etens\cor_\BS)$, which is in accordance with~\cite{GKS12}*{eq~(1.21)}. Moreover, $K_I\simeq\cor_I\,[n]$ up to the choice of an orientation on $\BBS^*$. 

The Fourier-Sato transform $\FS^\wedge$ and its inverse $\FS^\vee$ are  the functors
\eq\label{eq:FS2}
&&\xymatrix{
	\FS^\wedge\eqdot \cor_P\conv\cl\Derb(\cor_{\BS})\ar@<.5ex>[r]&\Derb(\cor_{\BS^*})\ar@<.5ex>[l]\cl \conv K_I\eqdot\FS^\vee
}
\eneq

\begin{theorem}[{see~\cite{SKK73}}]\label{th:FS1}
	The functor $\FS^\wedge$ and the functor $\FS^\vee$ are equivalences of categories quasi-inverse to each other.
\end{theorem}
We shall give a proof of this result at the same time as we shall prove Theorem~\ref{th:FS2} below.

Now, we consider the  $n$-sphere $\BBS^n$  of radius $1$   embedded in the Euclidian space   $\R^{n+1}$ and endowed  with its canonical Riemannian metric. 
Denoting by $\vvert\cdot\vvert$ the Euclidian norm on $\R^{n+1}$, the map 
\eqn
&&\R^{n+1}\setminus\{0\}\to\BBS^n,\quad x\mapsto x/\vvert x\vvert
\eneqn
identifies the topological sphere $\BS^n=(\R^{n+1}\setminus\{0\})/\R^+$ and the Euclidian sphere $\BBS^n$. 

The isomorphism $\R^n\simeq\R^{n*}$ induces the isomorphism $\BBS^n\simeq\BBS^{n*}$ and we shall identify these two spaces.
When there is no risk of confusion, we write for short  $\BBS\eqdot\BBS^n$.
Recall that (using the notations defined in~\eqref{hyp:riemann}):
\eqn
&&r_\iinj(\BBS^n)=\pi,\quad r_{\rconv}(\BBS^n)=\pi/2.
\eneqn
The next result is obvious and is also a corollary of Theorem~\ref{th:riemann1}. 

\begin{proposition}\label{pro:FS1}
	The metric space $\BBS$  satisfies~\eqref{hyp:dist1},~\eqref{hyp:dist2} and~\eqref{hyp:dist4} when choosing $\alpha_\BBS<\pi/2$.
\end{proposition}
In particular,   $\BBS^n$  admits a bi-thickening $\{\stL_b\}_{b\in\R}$. 

\begin{lemma}\label{le:FS1}
	For $0<a\leq b\leq\pi/2$, one has 
	$\cor_{\Delta^\circ_a}\conv\cor_{\Delta_b}\,[n]\simeq \cor_{\Delta_{b-a}}$. 
\end{lemma}
\begin{proof}
	Consider the diagram
	\eqn
	&&\xymatrix@R=5ex@C=4ex{
		&\BBS\times \BBS\times\BBS\ar[ld]_-{q_{12}}\ar[rd]^-{q_{23}}\ar[d]|-{q_{13}}&\\
		\Delta_a^\circ\subset\BBS\times\BBS&\BBS\times\BBS&\BBS\times\BBS\supset\Delta_b
	}\eneqn
	For $x_1,x_3\in\BBS$, set for short
	\eqn
	&&P^b_{x_3}=\Delta_b\cap(\BBS\times\{x_3\}),\quad I^a_{x_1}=\Delta_a^\circ\cap (\{x_1\}\times\BBS).
	\eneqn
	Denote by $\tw q_{13}$ the restriction of $q_{13}$ to 
	$\Delta_a^\circ\times_\BBS\Delta_b$. Then 
	\eqn
	\opb{\tw q_{13}}(x_1,x_3)=\{x_2\in\BBS;d_\BBS(x_1,x_2)<a,d_\BBS(x_2,x_3)\leq b\}.
	\eneqn
	In other words, $\opb{\tw q_{13}}(x_1,x_3)$ is the intersection of an open ball of radius $a$ and a closed ball of radius $b$ with $a\leq b$. 
	It follows that
	\eqn
	\rsect_c(I^a_{x_1}\times_{\BBS} P^b_{x_3};\cor_{\BBS\times\BBS\times\BBS})&=&
	\begin{cases}
		\cor\,[-n]&\text{ if $d_\BBS(x_1,x_3)\leq b-a$},\\
		0& \text{ otherwise}.
	\end{cases}
	\eneqn
\end{proof}

\begin{theorem}\label{th:FS2}
	The equivalence $\FS^\wedge$ given by {\rm Theorem~\ref{th:FS1}}  induces an isometry 
	\begin{equation*}
		(\Derb(\cor_\BBS),\dist_\BBS)\isoto (\Derb(\cor_{\BBS^*}),\dist_{\BBS^*}).
	\end{equation*}
\end{theorem}
\begin{proof}[Proof of both {\rm Theorems~\ref{th:FS1}} and~{\rm \ref{th:FS2}} ]
	Let us identify $\BBS^n$ and the dual sphere $\BBS^{n*}$. Then the sets $P$ and $I$ of~\eqref{eq:PI} may be also defined as:
	\eq\label{eq:PI2}
	&&P=\{(x,y)\in\BBS\times\BBS;d_\BBS(x,y)\leq\pi/2\},\quad I=\{(x,y)\in\BBS\times\BBS;d_\BBS(x,y)<\pi/2\}.
	\eneq
	Since $\cor_{\Delta_{\pi/2}}\simeq \cor_{\Delta_{\pi/4}}\conv \cor_{\Delta_{\pi/4}}$ we have 
	$\cor_P\simeq\stD_{\pi/2}$. (It was not possible to deduce directly this result form~\eqref{eq:PI2} since $\alpha_\BBS<\pi/2$.)
	Therefore $\cor_P\conv$  is an isometry  and the inverse of $\cor_P$ is given by 
	$\stD_{-\pi/2}$ which is isomorphic to $K_I$.
\end{proof}

\begin{remark}
	A similar result holds for the Radon transform on real projective spaces. 
\end{remark}

\section{The interleaving distance associated with a Hamiltonian isotopy}\label{sect:symp}

\subsection{General case}\label{subsect:negatisotop}
Let us briefly recall the main result of~\cite{GKS12}~{\S~3}. 
Consider a real $C^\infty$-manifold $X$, its cotangent bundle $\pi_X\cl T^*X\to X$ endowed with the Liouville form $\alpha_X$ and an open interval $I$ of $\R$ containing $0$. Set as above $\dT^*X=T^*X\setminus T^*_XX$,  where $T^*_XX$ is the zero-section, and still denote by $\pi_X\cl\dT^*X\to X$ the projection. When there is no risk of confusion, we may write $\pi$ instead of $\pi_X$. 

Assume to be given  a real $C^\infty$-function $h\cl\dT^*X\times I\to\R$  homogeneous of degree $1$ with respect to the fiber variable. Let $\Phi_h$ denote the flow associated with the  Hamiltonian vector field $H_h$. We assume that $\Phi_h$  is well-defined on the open interval $I\subset\R$. Hence, 
\eq\label{eq:hamiltisot}
&&\Phi_h\cl\dT^*X\times I\to\dT^*X
\eneq
and~\cite{GKS12}*{hypothesis (3.1)} is satisfied, that is, setting $\varphi_{h,t}=\Phi_h(\cdot,t)$,  $\varphi_{h,t}$ is a homogeneous  symplectic isomorphism of $\dT^*X$ for each $t\in I$ and $\varphi_{h,0}=\id_{\dT^*X}$. 
To $\Phi_h$, one associates 
\eqn
&&v_{\Phi_h}=\frac{\partial\Phi_h}{\partial t}\cl \dT^*X\times I\to T\dT^*X.
\eneqn
One recovers $h$ by $h=\langle \alpha_X,v_{\Phi_h}\rangle$.

Denote by $\Lambda_h\subset \dT^*X\times \dT^*X\times T^*I$
the smooth conic Lagrangian manifold associated with $\Phi_h$ (see~\cite{GKS12}*{Lem.~A.2}):
\eq\label{eq:lagrangianLh}
&&\Lambda_h=\{(\Phi_h(x,\xi,t),(x,-\xi),(t,-h(\Phi_h(x,\xi,t),t)));(x,\xi)\in\dT^*X,t\in I\}.
\eneq
The main result  \textcolor{blue}{of} loc.\ cit.\  (see~\cite{GKS12}*{Th.~3.7})  is the existence of an object 
$K^h\in\Derlb(\cor_{X\times X\times I})$  (denoted $K$ therein) characterized by the two properties:
\eq\label{eq:unicity}
&&\SSi(K^h)\subset \Lambda_h\cup T^*_{X\times X\times I}(X\times X\times I)\mbox{ and }K^h\vert_{\{t=0\}}\simeq\cor_\Delta.
\eneq
Now we assume that  
\eq\label{eq:hyphnottd}
\left\{\parbox{70ex}{
	$h$ is not time-depending, homogeneous of degree $1$ with respect to the fiber variable and the hamiltonian flow $\Phi$ 
	is well-defined on $\dT^*X\times \R$.
}\right.
\eneq
Note that since  $h$ is not time-depending,  the hamiltonian flow $\Phi$ 
is well-defined on $\dT^*X\times \R$ as soon as it is  well-defined on $\dT^*X\times I$ for some open interval $I$ containing $0$.

One has
\eq\label{eq:hamil1}
&&\phi_{h,a}\conv\phi_{h,b}=\phi_{h,a+b}.
\eneq
Therefore the object $K^h$ belongs to $\Derlb(\cor_{X\times X\times \R})$.

For $a\in \R$, we set $K^h_a=K^h\vert_{t=a}$.
\begin{lemma}
	Assuming~\eqref{eq:hyphnottd}, we have the isomorphisms 
	\eq\label{eq:hamil2}
	&& K^h_a\conv K^h_b\simeq K^h_{a+b}\mbox{ for }a,b\in \R. 
	\eneq
\end{lemma}
\begin{proof}
	By~\eqref{eq:hamil1}, the two isotopies $\{\Phi_{h,a} \circ \Phi_{h,t} \}_{t\in \R}$ and $\{\Phi_{h,a+t}\}_{t\in I}$ coincide.  
	Their associated kernels are respectively 
	$K^h_a \circ K^h$ and $\oim{T_a}(K^h)$, where $T_a$ is the translation $(x,x',t) \mapsto (x,x',t+a)$. These two kernels are micro-supported by $\Lambda$ and their restriction at $t=-a$ are isomorphic to $\cor_\Delta$. They are thus  isomorphic by the unicity of kernels satisfying~\eqref{eq:unicity} and restricting to $t=b$, we get~\eqref{eq:hamil2}.
\end{proof}

Now we assume
\eq\label{hyp:positive}
&&\left\{\parbox{75ex}{
	the function $h$ is non-positive.
}\right.\eneq
In the sequel, we denote by $(t;\tau)$ the coordinates on $T^\ast \R$.
Therefore, $\Lambda_h\subset  \dT^*X\times \dT^*X\times T_{\tau\geq0}^*\R$ and it follows from~\cite{GKS12}*{Prop.~4.8} that for 
$a\leq b\in\R$ there are natural morphisms 
\eqn
&&\rho_{a,b}\cl K^h_b\to K^h_a,
\eneqn
satisfying the compatibility conditions of Theorem~\ref{def:pshK}. Therefore we have:
\begin{theorem}\label{th:stKh}
	Assume to be given a real non-positive $C^\infty$-function $h\cl \dT^*X\to\R$  homogeneous of degree $1$ in the fiber variable such that the associated flow $\Phi_h$ is defined on $\dT^*X\times I$ for an open interval $I$ containing $0$.
	Then the family $\{K^h_a\}_{a\in\R}$ defines   a monoidal presheaf  $\stK^h$ on $(\R,+)$ with values in 
	$(\Derb(\cor_{X\times X}),\conv)$.
\end{theorem}
(Recall that for a monoidal presheaf $\stK$ on $(\R,+)$ one sets $\stK_a\eqdot\stK(a)$.)
\begin{remark}
	One shall not confuse the monoidal presheaf $\stK^h$, a presheaf on the monoidal ordered set $(\R_\geq,+)$ with values in  $\Derb(\cor_{X\times X})$ and $K^h$, an object of
	$ \Derlb(\cor_{X\times X\times\R})$. The object $K^h$ is explicitly calculated in~\cite{GKS12}*{Exa.~3.10, 3.11} for the cases of the Euclidean space
	and  the Euclidian sphere. 
\end{remark}

\begin{definition}\label{def:disth}
	We denote by  $\dist_h$  the pseudo-distance on $\Derb(\cor_X)$
	associated with  the  monoidal presheaf  $\stK^h$ \lp see Definition~\ref{def:convdist}\rp.
\end{definition}

\begin{remark}
	The notion of non-positive isotopy is due to~\cite{EKP06}. Let us also mention that several distances naturally appear in symplectic topology (see for example the recent paper~\cite{RZ20}). As far as we know, the pseudo-distance
	$\dist_h$  on sheaves on $X$  is new. 
\end{remark}

\subsection{The case of Riemannian manifolds}\label{sect:Riemann}

In this Section, we shall use some classical results of  Riemannian geometry, referring to~\cites{DC92, Chav06}. 

Consider a Riemannian manifold $(X,g)$ of class $C^\infty$ and denote by $d_X$ its associated distance. We assume

\eq\label{hyp:riemann}
&&\parbox{70ex}{
	$(X,g)$ is complete and has  a strictly positive convexity radius $r_{\rconv}$, hence  strictly positive injectivity radius $r_\iinj$. 
}  
\eneq
Recall that  $r_{\rconv} \leq \frac{r_\iinj}{2}$ (see \cite{Berger76}).  
\eq
&&\parbox{70ex}{
	For $(X,g)$  satisfying~\eqref{hyp:riemann}, we choose $0<\alpha_X<r_{\rconv}$.
}.\label{hyp:riemann3}
\eneq
Note that a compact Riemannian manifold satisfies hypothesis~\eqref{hyp:riemann}.

Consider the cotangent bundle $T^*X$ and its zero-section $T^*_XX$. The isomorphism $TX\isoto T^*X$ endows $T^*X$ with a metric and we denote by $\vvert\xi\vvert_x$ the norm of the vector $\xi\in T^*_xX$. 

For the reader's convenience, we recall some of the notations~\eqref{eq:BDG} and introduce some new ones:

\eq\label{not:openball}
&&\left\{ \ba{l}B_a(x_0)=\{x\in X;d_X(x_0,x)\leq a\},\\
B^\circ_a(x_0)=\{x\in X;d_X(x_0,x)< a\},\\
\Delta_a=\{(x_1,x_2)\in X\times X;d_X(x_1,x_2)\leq a\},\\
\Delta^\circ_a=\{(x_1,x_2)\in X\times X;d_X(x_1,x_2)<a\},\\
S_a(x_0)=\{x\in X; d_X(x_0,x)=a\},\\
B^*_X(r)=\{(x;\xi)\in T^*X;\vvert\xi\vvert_x<r\}, \\
S^*_X(r)=\{(x;\xi)\in T^*X;\vvert\xi\vvert_x=r\}.
\ea\right.
\eneq

We  also introduce the sets: 
\eq\label{eq:defZ}
&&\left\{ \ba{l}
I=]-r_{\iinj}, r_{\iinj}[,\quad I^+=]0,r_{\iinj}[,\quad I^-=\mathopen]-r_{\iinj}, 0[,\\
J=X\times X\times I,\quad J^\pm=X\times X\times I^\pm,\\
Z=\{ (x,y,t) \in J; d_X(x,y)\leq t<r_{\iinj}\},\\
\Omega^+= \{ (x,y,t) \in J;d_X(x,y)< t\},\\
\Omega^-= \{ (x,y,t) \in J; d_X(x,y)< -t\},\\
A=\{ ((x;\xi),t)\in T^*X\times I;\vvert\xi\vvert_x\leq t<r_{\iinj}\}
\ea\right.
\eneq
Let us recall the construction of the exponential map.
Consider the function
\eq\label{eq:hvertxi00}
&&f\cl  T^*X\to\R,\quad f(x,\xi)=-\frac{1}{2}\vvert\xi\vvert^2_x.
\eneq
Denote by  $X_f$ 
the Hamiltonian vector fields  of  $f$ and by  $\Phi_f$ the flows  associated to this vector fields.
In the literature (see {\em e.g.,}~\cite{MDS10}*{Exa.~1.1.23}, \cite{Pat99}*{p.~15}),  the flow $\Phi_f$ is known (via the isomorphism $TX\simeq T^*X$) as {\em the geodesic flow} of the Riemannian manifold $(X,g)$.

The exponential map  $e_f$, given by
\eqn
&& e_f(x,\xi,t)=\pi_X\circ\Phi_f(x,\xi,t),
\eneqn  
is well-defined  for $t\in\R$. 
The well-known theorem (see loc.\ cit.) which asserts that the geodesic flow 
coincides with the hamiltonian flow of the function $f$ may be translated as follows.

\begin{lemma}\label{le:geoflo}
	The map 
	\eq\label{eq:Ef} 
	&&E_f\cl T^*X\times I\to J=X\times X\times I,\quad E_f(x,\xi,t)=(e_f(x,\xi,1),x,t)
	\eneq
	is well-defined and  induces $C^\infty$-isomorphisms 
	\eqn
	&&B^*_X(r)\times\{t\}\simeq\Delta^\circ_r\times\{t\}\mbox{ for $r<r_{\iinj}$ and all $t$.}
	\eneqn
\end{lemma}

The proof of the next lemma is due to St{\'e}phane Guillermou. It is much  simpler than an earlier proof of ours.

\begin{lemma}\label{le:gu}
	Let $(X,g)$ be a Riemannian manifold satisfying~\eqref{hyp:riemann} and let $\alpha_X$ be as in~\eqref{hyp:riemann3}. 
	Let $x$ and $y$ in $X$ with $x\neq y$ and set $Z_a(x,y)=B^\circ_a(x)\cap B_a(y)$.  Then $\rsect(X;\cor_{Z_a(x,y)})\simeq0$. In other words, \eqref{hyp:dist2}(c) is satisfied.
\end{lemma}
\begin{proof}
	(i) We may assume
	\eq\label{eq:gu1}
	&&\left\{\parbox{70ex}{
		for any $x_1,x_2$ in $W$ with $x_1\neq x_2$, there exists a unique geodesic $l(x_1,x_2)\subset W$ with $x_1,x_2\in l(x_1,x_2)$,\\
		for $x_1,x_2,x_3$ in $W$, if $d(x_1,x_3)=d(x_1,x_2)+d(x_2,x_3)$ then $x_2\in l(x_1,x_3)$. 
	}\right.\eneq
	Let us introduce some notations:
	\eqn
	&&Z_a=Z_a(x,y),\\
	&&M=\{z; d(x,z) = d(y,z) \},\\
	&& M_x = \{z; d(x,z) < d(y,z) \},\quad M_y = \{z; d(x,z) > d(y,z) \},\\
	&&Z' = M_x \cap B_a(y),\quad Z'' = B^\circ_a(x) \cap \ol M_y.
	\eneqn
	Note that $Z_a=Z'\sqcup Z''$, $Z'$ is open in $Z_a$ and $Z''$ is closed in $Z_a$.
	
	\spa
	(ii) It follows from~\eqref{eq:gu1} that
	\eq\label{eq:gu2}
	&&\left\{\parbox{70ex}{
		for any geodesic $l(x,z)$, $l(x,z)\cap M$ has at most one point, and similarly with $l(y,z)$.
	}\right.\eneq
	Indeed, let $z_1,z_2\in l(x,z)\cap M$. Then $d(x,z_1)=d(x,z_2)+d(z_2,z_1)$ or 
	$d(x,z_2)=d(x,z_1)+d(z_1,z_2)$  or $d(z_1,z_2)=d(z_1,x)+d(x,z_2)$. Assume for example the first equality. Since  $z_1,z_2\in M$, we get 
	$d(y,z_1)=d(y,z_2)+d(z_2,z_1)$ which implies that the geodesic 
	$(y,z_1)$ contains $z_2$. Since there is at most one geodesic containing both $z_1$ and $z_2$, we find that 
	$y\in l(x,z)$ which implies $z_1=z_2$. 
	
	\spa
	(iii) Let us prove that 	$\rsect(X; \cor_{Z'}) \simeq 0$.
	Let   $p \cl B_a(y)\setminus\{y\}\to S_a(y)$ be the map which
	sends $z\in B_a(y)\setminus\{y\}$  to $p(z)\in l(y,z)\cap S_a(y)$. It follows from~\eqref{eq:gu2} 
	that the fibers of  $p$ intersect $Z'$ along
	a unique interval and this interval is half-open. 
	Since  $y\notin\ol Z'$, we have 
	$\rsect(X; \cor_{Z'})  \simeq\rsect(B_a(y); \cor_{Z'})  \simeq\rsect(B_a(y)\setminus\{y\}; \cor_{Z'})$. Moreover, $\rsect(B_a(y)\setminus\{y\}; \cor_{Z'})\simeq  \rsect(S_a(y);\reim{p}\cor_{Z'})\simeq0$.
	
	\spa
	(iv) Let us prove that $\rsect(X; \cor_{Z''}) \simeq 0$. 	
	Let $q\cl B_a(x)\setminus\{x\}\to S_a(x)$ be the map which
	sends $z\in B_a(x)\setminus\{y\}$  to $p(z)\in l(x,z)\cap S_a(x)$. It follows from~\eqref{eq:gu2} 
	that the fibers of  $q$ intersect $Z''$ along a unique interval and this interval is half-open. 
	
	Since  $x\notin\ol Z''$, we have 
	$\rsect(X; \cor_{Z''})  \simeq\rsect(B_a(x); \cor_{Z''})  \simeq\rsect(B_a(x)\setminus\{x\}; \cor_{Z''})$. Moreover, $\rsect(B_a(x)\setminus\{x\}; \cor_{Z''})\simeq  \rsect(S_a(x);\reim{q}\cor_{Z''})\simeq0$.
	
	\spa
	(v) The result then follows from the distinguished triangle $\cor_{Z'}\to\cor_{Z_a}\to\cor_{Z''}\to[+1]$.
\end{proof}

\begin{theorem}\label{th:riemann1}
	Let $(X,g)$ be a real Riemannian manifold satisfying~\eqref{hyp:riemann} and let $\alpha_X$ be as in~\eqref{hyp:riemann3}.  Then hypotheses~\eqref{hyp:dist1},~\eqref{hyp:dist2} and~\eqref{hyp:dist4} are satisfied. 
\end{theorem}

\begin{proof}
	(A) Let us prove~\eqref{hyp:dist1}.
	
	\spaa
	(a)--(i) Let $x_1$ and $x_2$ in $X$. Since $a,b \leq \alpha_X < r_\rconv$, the ball $B_a(x_1)$ and $B_a(x_2)$ are geodesically convex. Hence, their intersection is either empty or also geodesically convex and geodesically convex sets are contractible.
	
	\spa
	(a)--(ii) The closed and bounded subsets are compact by the  Hopf--Rinow Theorem. Therefore,  condition (ii) is satisfied.
	
	\spa
	(a)--(iii) Let us prove that for $(x_1,x_3) \in \Delta_{a+b}$, there exists $x_2 \in X$ such that $d_X(x_1,x_2)\leq a$ and  $d_X(x_2,x_3)\leq b$. Without loss of generality we can assume that $d_X(x_1,x_3)=a+b$.
	Since $X$ is complete, it follows from the Hopf--Rinow Theorem that $x_1$ and $x_3$ can be joined by a minimal geodesic $\gamma \colon [0,1] \to X$. 
	Then $d(x_1,\gamma(t))$ will take all values between $0$ and $a+b$. Let $t_2\in[0,1]$ such that $d(x_1,\gamma(t_2))=a$. Since $\gamma$ is also minimal on every subinterval of $[0,1]$ it is minimal on $[t_2,1]$. Then, $d_X(x_2,x_3)=b$.
	
	\spaa
	(B) Let us prove~\eqref{hyp:dist2}(b).
	The set $\Omega^+$ is, in a neighborhood of $\Delta\times\{0\}$ and locally in $X\times X\times \R$, $C^\infty$-isomorphic to
	the open set $\{(x,\xi,t); ||\xi||_x < t\}$.
	By the Morse lemma with parameters (see~\cite{Ho85}*{Lem.~C.6.1 and its proof}) this last set is locally topologically convex since, in a local chart, it is isomorphic to a constant cone $\{((x;\xi),t);\vvert \xi\vvert<t\}$ associated with the standart Euclidian metric. 
	
	\spaa
	(C) Let us prove~\eqref{hyp:dist2}(a).
	By Lemma~\ref{le:geoflo}, we are reduced to prove the result after replacing $\Delta_a$ with $B_X^*(a)$ in which case the proof is similar to (B).
	
	\spaa
	(D) The hypothesis~\eqref{hyp:dist2}(c) is satisfied thanks to  Lemma~\ref{le:gu}.
	
	\spaa
	(E)  The hypothesis~\eqref{hyp:dist4} follows from Lemma~\ref{lem:add5}. Indeed, 
	the distance function $f\eqdot d_X\cl X\times X\to\R$ is of class $C^\infty$ on $W\eqdot \Delta^\circ_a\setminus\Delta$ for $a\leq\alpha_X$ and 
	we are reduced to check that for any given $y\in X$, the differential of the function $x\mapsto g(x)=d_X(y,x)$ does not vanish for $0<d_X(x,y)<\alpha_X$. By composing with  the exponential map, we are reduced to prove the same result on $T^*_yX$ in which case it is clear.
\end{proof}

\begin{notation}\label{not:stkda}
	We shall denote by  $a\mapsto\stKd_a$, $a\in\R$ the  bithickening of the diagonal given by Theorems~\ref{th:riemann1}
	and Proposition~\ref{pro:bithick1}.
\end{notation}

\subsection{Comparison of the two kernels on Riemannian manifolds}\label{subsect:hamiltonian}

In this subsection,
$(X,g)$ denotes a Riemannian manifold  with associated distance $d_X$. We shall always assume~\eqref{hyp:riemann}.

Recall the function $f$ and the flow $\Phi_f$ defined in~\eqref{eq:hvertxi00}, and consider the function
\eq\label{eq:hvertxi}
&&h\cl \dT^*X\to\R,\quad h(x,\xi)=-\vvert\xi\vvert_x.
\eneq
Denote by $X_h$ 
the Hamiltonian vector fields  of $h$  and by  $\Phi_h$  the flow  associated to this vector fields.
Since $h$ is homogeneous of degree $1$ in $\xi$ and  $f$ is homogeneous of degree $2$ in $\xi$, we have for  $\lambda >0$

\eq\label{eq:phihomog}
&&\left\{\ba{l}
\Phi_h(x,t;\lambda \xi) = \lambda \cdot \Phi_h(x, t;\xi),\\
\Phi_f(x,t;\lambda \xi) = \lambda \cdot \Phi_f(x, \lambda t;\xi).
\ea\right.
\eneq
(Of course, in the formula above, $\lambda$ acts on the fiber variables.)

Since $f = -\frac12 h^2$, the Hamiltonian vector fields of $f$ and $h$ are related by $X_f = -h X_h = \vvert \xi\vvert X_h$.  In particular,   
we see that $X_f$ and $X_h$ are tangent to the unit co-sphere $S^*_X(r)$ and their restrictions to $S^*_X(1)$ coincide.  It 
follows that $\Phi_h(x,t;\xi) = \Phi_f(x,t;\xi)$ if $\vvert \xi\vvert = 1$ and, by homogeneity, using~\eqref{eq:phihomog}
\eq\label{eq:phih}
\Phi_h(x,t;\xi)  = \vvert\xi\vvert_x \cdot  \Phi_f(x,t; \frac{\xi}{\vvert\xi\vvert_x})
=  \vvert\xi\vvert_x \cdot \Phi_f(x,1;\frac{t}{\vvert\xi\vvert_x}\xi ) \mbox{ for }  \xi\not=0.   
\eneq
By the hypothesis~\eqref{hyp:riemann},  we get
\begin{lemma}
	Hypothesis~\eqref{eq:hyphnottd} is satisfied for $h$.
\end{lemma}
Denote as above by  $\Lambda_h$ the Lagrangian manifold given by~\eqref{eq:lagrangianLh}. 
One has
\eq\label{eq:lagrangianLhnt}
&&\Lambda_h=\{(\Phi_h(x,\xi,t),(x,-\xi),(t,\vvert\xi\vvert_x));(x,\xi)\in\dT^*X,t\in \R\}.
\eneq
Denote  by $K^h$ the quantization of $\Lambda_h$ and by 
$\stk^h$ the monoidal presheaf on $(\R,+)$ with values in 
$\Derb(\cor_{X\times X},\conv)$  associated with $K^h$ constructed  in Theorem~\ref{th:stKh} and denote by $\stKd$ the monoidal presheaf associated with the good metric space  $(X,d_X)$ (see Theorem~\ref{th:riemann1} and Notation~\ref{not:stkda}).

With Notations~\eqref{eq:defZ}, the distinguished triangle~\eqref{eq:GKSdt} reads as
\eq\label{eq:GKSdt2}
&&\cor_{\Omega^-}\tens\opb{q_2}\omega_X\to K^d\to\cor_{Z}\to[+1].
\eneq

\begin{lemma}\label{lem:geodflow}
	Assume~\eqref{hyp:riemann}. 
	One has 
	$\Lambda_h \cap T^*J^+ = \dot\SSi(\cor_{Z})\cap T^*J^+$.
\end{lemma}
\begin{proof}
	(i) Recall that
	\eq\label{eq:lagrangianLh2}
	&&\Lambda_h=\{(\Phi_h(x,\xi,t),(x,-\xi),(t,-h(\Phi_h(x,\xi,t));(x,\xi))\in\dT^*X,t\in I\}.
	\eneq
	In particular, 
	\begin{align*}
		\pi_{J^+}(\Lambda_h\cap T^*J^+)
		&=E_f(\{\vvert\xi\vvert_x\leq t\})  = \partial\Omega^+.
	\end{align*}
	
	\spa
	(ii) The set $\partial\Omega^+$ is a smooth
	hypersurface of $J^+$ and it follows from~\cite{KS90}*{Prop.~8.3.10} that 
	$\Lambda_h\cap T^*J^+$ is one half of
	$\dT^*_{\partial\Omega^+}J^+$.    
	Since $\Lambda_h\subset\{\tau\geq0\}$, $\Lambda_h$ is the interior conormal to $\partial\Omega^+$.
\end{proof}
Denote by $j\cl J^+\into J$ the open embedding. 
\begin{lemma}\label{le:2ker1}
	One has $\cor_Z\simeq \roim{j}\opb{j}\cor_{Z}$.
\end{lemma}
\begin{proof}
	One has $\cor_{\Omega^+}\simeq \eim{j}\opb{j}\cor_{\Omega^+}$. Applying the duality functor $\RD'_{X\times X\times\R}$ we get the result by Lemma~\ref{lem:add2}. (Recall that, setting $M=X\times X\times\R$,   $\RD'_M\circ\eim{j}\simeq\roim{j}\circ\RD'_M$. )
\end{proof}

In the proof of the next lemma, we shall use the operation $\widehat{+}$ defined in~\cite{KS90}*{\S~6.2}.
\begin{lemma}\label{le:2ker2}
	One has 
	\banum
	\item
	$\SSi(\cor_Z)\cap \opb{\pi}(X\times X\times\{0\})\subset\{(x,x,0;\xi,-\xi,\tau);\tau\geq\vvert\xi\vvert_x\}$,
	\item 
	One has $\SSi(\cor_{\Omega^-})\cap \opb{\pi}(X\times X\times\{0\})\subset\{(x,x,0;\xi,-\xi,\tau);\tau\geq\vvert\xi\vvert_x\}$.
	\eanum
\end{lemma}
\begin{proof}
	(a) Recall~\eqref{eq:lagrangianLhnt}. We have in a neighborhood of $t=0$
	\eq
	&&\Lambda_h=\{( x- \frac{t}{\vvert\xi\vvert_x}\xi +t^2\epsilon(x,t,\xi),  x,t;\xi+t\eta(x,t,\xi), -\xi,\vvert\xi\vvert_x);(x,\xi)\in\dT^*X,t\in \R\}.
	\eneq
	This implies 
	\eqn
	&&(\Lambda_h\cap T^*J^+)\widehat{+}\{(x,y, 0;0,0,\tau\geq0)\}\subset \{(x,x,0;\xi,-\xi,\tau);\tau\geq\vvert\xi\vvert_x\}.
	\eneqn
	To conclude,  apply~\cite{KS90}*{Th.~6.3.1} together with Lemmas~\ref{lem:geodflow} and~\ref{le:2ker1}.
	
	\spa
	(b) follows from (a) by applying the duality functor (using Lemma~\ref{lem:add2}) together with $\oim{v}$ where $v$ is the map $(x,y,t)\mapsto(x,y,-t)$. 
\end{proof}

\begin{lemma}\label{le:2ker3}
	Let $p=(x,x,0;\xi,-\xi,\tau)$ with $\tau>\vvert\xi\vvert_x$. Then 
	\banum
	\item
	the natural morphism 
	$\cor_Z\to\cor_{\Delta\times\{0\}}$ is an isomorphism in $\Derb(\cor_J;p)$. 
	\item 
	the natural morphism 
	$\cor_{\Delta\times\{t=0\}}\tens\opb{q_2}\omega^{\otimes-1}_X[-1]\to \cor_{\{d_X(x,y)<-t\}}$
	is an isomorphism in $\Derb(\cor_J;p)$. 
	\eanum
\end{lemma}
\begin{proof}
	(a) 
	Similarly as in part (C) of the proof of Theorem~\ref{th:riemann1}, the set $Z$ is, in a neighborhood of $\Delta\times\{0\}$ and locally 
	on $X\times X\times\R$, $C^\infty$-isomorphic to the set $A$ of~\eqref{eq:defZ}. 
	We are thus reduced to prove a similar result with $Z$ and $\Delta\times\{0\}$ replaced with $A$ and $T^*_XX\times\{0\}$. 
	In this case, the result follows from Lemma~\ref{le:wk} below.

	\spa
	(b) follows from (a) by applying the duality functor, using Lemma~\ref{lem:add2}.
\end{proof}

\begin{lemma}\label{le:wk}
	Let $E$ be a vector bundle over $X$ and let  $\gamma\subset E$ be a closed convex proper cone containing the zero-section $X$. 
	Let $p\in T^*E\times_EX$ with $p\in\Int(\gamma^\circ)$. Then the natural morphism $\cor_\gamma\to\cor_X$ is an isomorphism in 
	$\Derb(\cor_E;p)$.
\end{lemma}
\begin{proof}
	We may assume that $E=X\times \BBV$ for a real vector space $\BBV$. Let us choose local coordinates 
	on $X$ and identify $T^*\BBV$ with $\BBV\times\BBV^*$. Then $p=((x;\xi),(0,\eta))\in T^*X\times\BBV\times\BBV^*$. 
	By~\cite{KS90}*{Lem.~3.7.10}, the Fourier-Sato transform interchanges the two objects $\cor_\gamma$ and  $\cor_{X\times\{0\}}$ of  
	$\Derb(\cor_{E})$ with the two objects  $\cor_{\Int\gamma^\circ}$ and  $\cor_{E^*}$ of $\Derb(\cor_{E^*})$. Hence, applying Th.~5.5.5 and formula (5.5.6) of loc.\ cit., we are reduced to prove that the natural morphism 
	$\cor_{\Int\gamma^\circ}\to\cor_E$ is an isomorphism in $\Derb(\cor_{E^*};q)$ with 
	$q=((x;\xi),(\eta,0))\in T^*X\times\BBV^*\times\BBV$, which is obvious since the two sheaves are isomorphic in a neighborhood of any point 
	$(x,\eta)\in X\times \Int\gamma^\circ$.
\end{proof}

Recall  the sheaf  $K^d$ constructed in Theorem~\ref{th:bithick3} and the monoidal presheaf $\stKd$.
\begin{theorem}
	Let $(X,g)$ be a  complete Riemannian manifold satisfying~\eqref{hyp:riemann}. Then
	\banum
	\item
	One has the isomorphism $K^h \vert_{J} \simeq K^d\vert_{J}$.
	\item
	the two monoidal presheaves  $\stK^h$ and $\stKd$ are isomorphic. 
	\eanum
\end{theorem}

\begin{proof}
	(i) 	Of course, (b) follows from (a). By the unicity  result in~\cite{GKS12}*{Prop.~3.2~(iii)}, it remains to prove that 
	\eq\label{eq:SSiKh}
	&&\dot\SSi(K^d)\subset\Lambda_h. 
	\eneq
	(ii) It follows from the distinguished triangle~\eqref{eq:GKSdt2} that $K^d\vert_{J^+}\simeq\cor_{Z}\vert_{J^+}$ and it then follows from
	Lemma~\ref{lem:geodflow} that~\eqref{eq:SSiKh} is true on $J^+$. Moreover, 
	$\dot\SSi(K^d\vert_{J^-})= v(\dot\SSi(K^d\vert_{J^+}))$
	where $v$ is the map $(x,y,t;\xi,\eta,\tau)\mapsto(y,x,-t;\eta,\xi,\tau)$.
	Since $v(\Lambda_h)=\Lambda_h$, we get that~\eqref{eq:SSiKh} is true on $J^-$. 
	
	\spa 
	(iii) One has $\SSi(K^d)\cap \opb{\pi}(X\times X\times\{0\})\subset\{(x,x,0;\xi,-\xi,\tau);\tau\geq\vvert\xi\vvert_x\}$  thanks to Lemma~\ref{le:2ker3}.
	The natural morphism $\psi\cl \cor_Z\to \cor_{\Omega^-}\tens\opb{q_2}\omega_X[+1]$ is an isomorphism by 
	Lemma~\ref{le:2ker3}. This implies~\eqref{eq:SSiKh}.
\end{proof}

\providecommand{\bysame}{\stLeavevmode\hbox to3em{\hrulefill}\thinspace}

\vspace*{1cm}
\noindent
\begin{tabular}{cc}
\parbox[t]{14em}
{\scriptsize{
		Fran\cc ois~Petit \\
		Université de Paris\\ 
		CRESS, INSERM, INRA\\ 
		F-75004 Paris France\\
		e-mail address: francois.petit@u-paris.fr\\}}
		&
\parbox[t]{14em}{\scriptsize{
		Pierre Schapira\\
		Sorbonne Universit{\'e}, CNRS IMJ-PRG\\
		4 place Jussieu, 75252 Paris Cedex 05 France\\
		e-mail: pierre.schapira@imj-prg.fr\\
		http://webusers.imj-prg.fr/\textasciitilde pierre.schapira/
}}
\end{tabular}
\end{document}